\documentclass[11pt]{article}
\usepackage[tbtags]{amsmath}
\usepackage{cases}
\usepackage{mathrsfs}
\usepackage{amsfonts}
\usepackage{authblk}
\usepackage[colorlinks,linkcolor=blue, citecolor=blue]{hyperref}

\usepackage{amsmath,amsthm,amssymb}
\usepackage{cite}
\usepackage{geometry}
\usepackage{color}

\theoremstyle{plain}
\newtheorem{theorem}{Theorem}[section]
\newtheorem{lemma}{Lemma}[section]
\newtheorem{remark}{Remark}[section]
\newtheorem{proposition}{Proposition}[section]
\newtheorem{corollary}{Corollary}[section]
\newif \ifLastSection \LastSectionfalse

\numberwithin{equation}{section}

\allowdisplaybreaks
\allowdisplaybreaks[4]

\begin{document}
	
	\title{{\bf {Boundary effect on asymptotic behaviour of solution to the hyperbolic-parabolic chemotaxis system}}}

	\author{Nangao Zhang\thanks{Corresponding author. \authorcr Email addresses:
			mathngzhang@163.com
			.}}
	\affil{ \normalsize School of Mathematics and Statistics, Wuhan Textile University, Wuhan, 430200, P.R. China}
	
	\date{}
	
	\maketitle
	
	\textbf{{\bf Abstract:}} 
	This paper  is concerned with the global existence and stability of solution to the quasi-linear hyperbolic-parabolic chemotaxis system on the half-line, which was proposed in \cite{Ambrosi1} to primarily describe the formation of coherent vascular networks observed in vitro experiment. 
	Under two different boundary conditions, we first establish the existence and uniqueness of the solution of the asymptotic state equation (the so-called nonlinear diffusion wave), and then prove that the solution of the original equation is nonlinearly asymptotically stable. 
	Additionally, we obtain the convergence rates for both types of boundary conditions.

	\bigbreak \textbf{{\bf Key words}:}  Hyperbolic-parabolic chemotaxis system, half-line, nonlinear diffusion wave.
	
	\bigbreak \textbf{{\bf AMS Subject Classification}:} 35A01, 35B40, 35B45, 35K57, 92C17.
	
	\tableofcontents
	\section{Introduction}\label{sec1}
		Angiogenesis refers to the process of forming new blood vessels from existing capillaries or posterior capillaries, which is a complex process involving many cells and molecules. 
In recent years, due to its importance in pathological and physiological conditions, this phenomenon has become the subject of many experimental studies. 
In order to study this complex phenomenon and accurately describe the key features of human vascular experiments in vitro, Ambrosi et al. (cf.\cite{Ambrosi1, Ambrosi2}) proposed the following quasi-linear hyperbolic parabolic syatem:
	\begin{equation}\label{sanweiqngxing}
	\left\{
	\begin{array}{lr}
		\partial_t\rho+\nabla\cdot\left(\rho u\right)=0, \\[2mm]
	\partial_t\left(\rho u\right)+\nabla\cdot\left(\rho u\otimes u\right)+\nabla p(\rho)=\mu\rho\nabla\phi-\alpha\rho u,\\[2mm]
	\partial_t\phi=D\Delta\phi+a\rho-b\phi,
	\end{array}
	\right.
\end{equation}
	where $(x,t) \in \mathbb{R}^n\times \mathbb{R}^+ ~ (n\geq 1)$.  
	The unknown functions  $\rho=\rho(x,t)\geq0$, $u=u(x,t) \in \mathbb{R}^n$ and $ \phi=\phi(x,t)\geq0$ represent the density, velocity of endothelial cells and the concentration of the chemoattractant respectively.
	 The density-dependent quantity $p(\rho)$ is the pressure function which satisfies $p'(\rho)>0$ for $\rho>0$.
	In addition, $\alpha>0$ is the damping coefficient, $\mu>0$ is the strength of cells response to the chemoattractant concentration gradient, $D>0$ is the diffusion coefficient of the chemoattractant, and $a>0$ and $b>0$ denote the secretion and death rates of the chemoattractant respectively.

We first give a brief review of the current research status of system \eqref{sanweiqngxing}. 
As an important class of biological models, this system has attracted widespread attention from numerous researchers since its proposal. 
However, due to the particularity of its structure, studies on this system remain very limited and far from satisfactory. 
Di Russo et al. \cite{Russo1,Russo2} proved that in multidimensional cases, system \eqref{sanweiqngxing} admits global strong solutions without vacuum by combining characteristics of the hyperbolic part and the parabolic part; these solutions converge to constant states, and exhibit algebraic decay if the perturbations are sufficiently small. 
Chavanis et al. \cite{sire} studied the relationship between \eqref{sanweiqngxing} and the standard Keller-Segel system through asymptotic analysis, suggesting that the solutions of the system share the same limit under the strong friction limit. 
Di Francesco et al. \cite{Frances} rigorously justified the conclusions in \cite{sire} by using energy methods and compensated compactness tools. 
Later, Hong et al. \cite{ZAwang2} investigated the existence and stability of phase transition steady states for \eqref{sanweiqngxing} on the half-line $\mathbb{R}^+$ with Dirichlet boundary conditions. 
Recently, Liu et al. (cf. \cite{Dong,QQL,Liu-Peng-wang-2022}) proved that for the Cauchy problem of equation \eqref{sanweiqngxing} in three or one spatial dimension, the solution exists globally and tends respectively toward a linear diffusion wave or a nonlinear diffusion wave. Liu et al. \cite{Liu-Peng-wang-2024} studied the relaxation limit problem of solution in three-dimensional space by using energy method.
For some results concerning numerical simulations of this system, refer to \cite{CCCC, cwshu, yalihanshujiashe,ZAwang1} and references therein.

In physics, damping effect usually leads to nonlinear diffusion in dynamic system. The strict mathematical proof of this inference was first discovered by Hsiao and Liu \cite{Hsiao-Liu1992} in their research on damped compressible Euler equations in the whole space. Subsequently, by applying more detailed energy estimates, Nishihara et al. \cite{Nishihara1996, Nishihara2000}  succeeded in improving the convergence rates to be optimal for the same problem. 
For the case on the half-line, Nishihara and Yang \cite{Nishihara-Yang1999} also obtained similar conclusions. Since then,  many researchers have made a series of improvements from different aspects, refer to \cite{Huang-Marcati-Pan2005, Zhu-Jiang2009, Luo2016,Marcati-Mei2000,Matsumura-Nishihara2023,Mei2005,Mei2010, Zhang2023} and the references therein. 
The purpose of this paper is to study this kind of diffusion phenomenon in quasi-linear hyperbolic-parabolic equations on a half-line.

To demonstrate our ideas, we use variables $(\rho,m,\phi)(x,t)$ instead of $(\rho,u,\phi)(x,t)$, where  $m=\rho u$  denotes the momentum of cells, and then rewrite the equations  \eqref{sanweiqngxing} for $(x,t)\in\mathbb{R}^+\times\mathbb{R}^{+}$ as the following form:	
		\begin{equation}\label{1}
		\left\{
		\begin{array}{lr}
			\rho_t+m_x=0, \\[2mm]
			m_t+\left(\frac{m^2}{\rho}\right)_x+p(\rho)_x=\mu\rho\phi_x-\alpha m,\\[2mm]
			\phi_t=D\phi_{xx}+a\rho-b\phi.
		\end{array}
		\right.
	\end{equation}
 The initial data is given by
	\begin{equation}\label{3}
		(\rho,m,\phi)(x,t)|_{t=0}=(\rho_0,m_0,\phi_0)(x)\to (\rho_{+},m_{+},\phi_{+}), \;\;\;
		\mbox{as} \;\;\;x \to +\infty,
	\end{equation}
with the following boundary conditions
\begin{equation}\label{3a}
 m(0,t)=0, \ \ \ \phi_x(0,t)=0,
 \end{equation}
or 
\begin{equation}\label{3b}
 m_{x}(0,t)=0,\ \ \ \phi(0,t)=\frac{a}{b}\rho_0(0).	
 \end{equation}
Here $\rho_{+}>0$, $m_{+}$ and $\phi_{+}$ are constants. 

According to the method of asymptotic analysis, the solution $\left(\rho,m,\phi\right)(x,t)$ of the initial-boundary value problem (IBVP) \eqref{1}--\eqref{3a} or \eqref{1}--\eqref{3} and \eqref{3b} is expected to converge to the nonlinear diffusion wave $\left(\bar\rho,\bar m,\bar\phi\right)(x,t)$, which is the solution to the following system:
	\begin{equation}\label{6}
		\left\{
		\begin{array}{lr}
			\bar\rho_t+\bar m_x=0,      \\[2mm]
			p(\bar\rho)_x=\mu\bar\rho\bar\phi_x-\alpha\bar m, \\[2mm]
			a\bar\rho=b\bar\phi,\\
		\end{array}
		\right.\;\;\;{\rm or \ equivalently,}\;\;\;\;\;\;
		\left\{
		\begin{array}{lr}
			\bar\rho_t-\frac{1}{\alpha}q(\bar\rho)_{xx}=0,     \\[2mm]
			\bar m=-\frac{1}{\alpha}q(\bar\rho)_x, \\[2mm]
			\bar\phi=\frac{a}{b}\bar\rho, \\
		\end{array}
		\right.
	\end{equation}
	with 
	\begin{equation}\label{7}
	\left\{
		\begin{array}{lr}
		\bar\rho_x(0,t)=\bar m(0,t)=\bar\phi_x(0,t)=0,\\[2mm]
		(\bar\rho,\bar m,\bar\phi)(x,t)\to\left(\rho_{+},0,\frac{a}{b}\rho_{+}\right),\;\;\;\mbox{as}\;\;\;x\to+\infty,
		\end{array}
		\right.
	\end{equation}
	or 
	\begin{equation}\label{8}
		\left\{
		\begin{array}{lr}
		\bar{\rho}(0,t)=\rho_0(0), \bar{m}_{x}(0,t)=0,\bar\phi(0,t)=\frac{a}{b}\rho_0(0),\\[2mm]
		(\bar\rho,\bar m,\bar\phi)(x,t)\to\left(\rho_{+},0,\frac{a}{b}\rho_{+}\right),\;\;\;\mbox{as}\;\;\;x\to+\infty.
		\end{array}
		\right.
	\end{equation}
	Here we denote
	\begin{equation}\notag
		q(\bar\rho):=p(\bar\rho)-\frac{a\mu}{2b}\bar\rho^2.
	\end{equation}

With the above analysis, the first problem we face next is to give the existence and decay properties of solution to system \eqref{6}--\eqref{7} or \eqref{6} and \eqref{8}.
For the first type of boundary, we will provide it through the energy method (see Lemma \ref{l3.1}); while for the second type of boundary, we will directly present it based on the definition of the self-similar solution (see \eqref{14.1}--\eqref{14.2}).
Secondly, it can be observed that there exist some gaps between the solution of the original equation and the nonlinear diffusion wave at $x=+\infty$.
Therefore, we need to construct a set of correction function $(\hat\rho,\hat m,\hat\phi)(x,t)$ to eliminate these gaps. 
To achieve this goal, through careful observation, we first analyze the behavior of solutions to \eqref{1}--\eqref{3a} or \eqref{1}--\eqref{3} and \eqref{3b} at the far field,
and then guess that the correction function $(\hat\rho,\hat m,\hat\phi)(x,t)$ should satisfies a new system. 	
By skillfully giving the initial value and boundary conditions of the new system, we obtain $(\hat\rho,\hat m,\hat\phi)(x,t)$ (see Section \ref{sec2} for more details). 
Consequently, we are able to prove the stability of the nonlinear diffusion wave for \eqref{1}--\eqref{3a} or \eqref{1}--\eqref{3} and \eqref{3b}.
	
	The arrangement of the present paper is as follows.
	In Section \ref{sec2}, we reformulate \eqref{1}--\eqref{3} and give the main results.
	We prove Theorem \ref{thm1} and Theorems \ref{thm2}--\ref{thm3} in Section \ref{s3} and Section \ref{s4} respectively.

\vspace{3mm}
	
	{\bf Notations.}
In the following, $C$ and $c$ denote the generic positive constants independent of the $x$ and $t$. For functions $f$  and $g$, $f\lesssim g$ denotes $f \leq C g$.
$L^{p}=L^{p}(\mathbb{R}^{+}) ~ (1\leq p\leq \infty)$ is an usual Lebesgue space with the norm
\begin{equation}\notag
 \|f\|_{L^{p}}=\left(\int_{\mathbb{R}^{+}}|f(x)|^{p}{\rm d}x\right)^{\frac{1}{p}}, ~~~ 1\leq p< \infty~~~\mbox{and} ~~~ \|f\|_{L^{\infty}}=\sup \limits_{\mathbb{R}^{+}}|f(x)|.
\end{equation}
For any integer $l \geq 0$, $H^{l}$ denotes the usual $l$-th order Sobolev space on $\mathbb{R}^{+}$ with its norm
\begin{equation}\notag
 \|f\|_{l}=\left(\sum_{j=0}^{l}\left\|\partial_{x}^{j} f\right\|^{2}\right)^{\frac{1}{2}}, ~~~ \quad\|\cdot\|=\|\cdot\|_{0}=\|\cdot\|_{L^{2}}.
\end{equation}

	\section{Reformulation of the  problem and main results}\label{sec2}

In this section, we first investigate the exact difference between the solution $\left(\rho,m,\phi\right)(x,t)$ to \eqref{1}--\eqref{3a} or \eqref{1}--\eqref{3} and \eqref{3b} and the diffusion wave at $x\to+\infty$.
Then, as shown in \cite{Hsiao-Liu1992,Nishihara-Yang1999}, we construct the correction function to deal with this difference so that the estimation can be done in $L^2$-framework. 
Finally, we reformulate the problem and give the main results of this paper.

\subsection{Initial-boundary value problem with boundary condition \eqref{3a}}

In this case, as previously analyzed, the smooth solution of IBVP \eqref{1}--\eqref{3a} is expected to converge to the solution of \eqref{6}--\eqref{7}.
The existence and decay properties of the solution to system \eqref{6}--\eqref{7} will be given in the Lemma \ref{l3.1}. 
Next, we are going to construct a group of correction functions.
Let's investigate the behavior of the solution to
	\eqref{1}--\eqref{3} at $x=+\infty$,
	then we can know how big the gaps are between the original solution and the diffusion wave.
	
	Taking the limits of \eqref{1} with respect to $x$,
	and noting that $m_x$, $\left(\frac{m^2}{\rho}\right)_x$, $p(\rho)_x$,
	$\mu\rho\phi_x$,
	$D\phi_{xx}$ will vanish at $x=+\infty$,
	then we have
	\begin{equation}\notag
		\left\{
		\begin{array}{lr}
			\displaystyle\frac{{\rm d}}{{\rm d}t}\rho(+\infty,t)=0,\\[2mm]
			\displaystyle\frac{{\rm d}}{{\rm d}t}m(+\infty,t)=-\alpha m(+\infty,t),\\[2mm]
			\displaystyle\frac{{\rm d}}{{\rm d}t}\phi(+\infty,t)=a\rho(+\infty,t)-b\phi(+\infty,t),\\[2mm]
			(\rho,m,\phi)(+\infty,0)=\left(\rho_0,m_0,\phi_0\right)(+\infty)=\left(\rho_{+},m_{+},\phi_{+}\right),
		\end{array}
		\right.
	\end{equation}
	By direct calculation, we get
	\begin{equation}\label{9}
		\lim_{x\to+\infty}\left(\rho,m,\phi\right)(x,t)=(\rho,m,\phi)(+\infty,t)=\left(\rho_{+},\;m_{+}{\rm e}^{-\alpha t},\;d_{+}{\rm e}^{-bt}+\frac{a}{b}\rho_{+}\right),
	\end{equation}
	where $d_+$ is defined as
	\begin{equation}\label{9a}
		d_{+}:=\phi_{+}-\frac{a}{b}\rho_{+}.
	\end{equation}
	From \eqref{7} and \eqref{9},  it holds that
	\begin{equation}\label{2.44444}
		\left\{
		\begin{array}{lr}
			|\rho(+\infty,t)-\bar\rho(+\infty,t)|=0	,\\[2mm]
			|m(+\infty,t)-\bar m(+\infty,t)|=|\;m_{+}{\rm e}^{-\alpha t}|\neq 0	,\\[2mm]
			|\phi(+\infty,t)-\bar\phi(+\infty,t)|=|d_{+}{\rm e}^{-bt}|\neq 0	.
		\end{array}
		\right.
	\end{equation}
	So, there are some gaps between the original solution and the corresponding asymptotic profile such that 
	\[
	m(x,t)-\bar m(x,t)\notin L^2(\mathbb{R}^+) \;\;\;{\rm and}	\;\;\;\phi(x,t)-\bar \phi(x,t)\notin L^2(\mathbb{R}^+).\]
	To eliminate these gaps, we need to introduce correction function $(\hat\rho,\hat m,\hat\phi)(x,t)$.
	By observing the structure of
	the original solutions at $x=+\infty$, we  technically construct the correction functions $(\hat\rho,\hat m,\hat\phi)(x,t)$ satisfying the following equations for $(x,t)\in\mathbb{R}^+\times\mathbb{R}^{+}$:
	\begin{equation}\label{11}
		\begin{cases}
\hat\rho_t+\hat m_x=0,     \\
			\hat m_t=-\alpha\hat m, \\
			\hat\phi_t=a\hat\rho-b\hat\phi,
\end{cases}
\quad \text{with} \quad
\begin{cases}
\hat\rho(x,t)\to0,      \\
			\hat m(x,t)\to m_{+}{\rm e}^{-\alpha t},  \\
			\hat\phi(x,t)\to d_{+}{\rm e}^{-bt},  & \mbox{as}\;\;\;x\to+\infty.\\
\end{cases}
	\end{equation}
Inspired by \cite{Nishihara1996, Hsiao-Liu1992}, let's define the initial data for \eqref{11} as follows:
\begin{equation}\label{12}
		\left\{
		\begin{array}{lr}
			\hat m(x,0)=m_+\displaystyle\int_{0}^{\varepsilon_{0}x}q_0(y){\rm d}y, \\[2mm]
			\hat\phi(x,0)=d_+\displaystyle\int_{0}^{\varepsilon_{0}x}q_0(y){\rm d}y,
		\end{array}
		\right.
	\end{equation}
	where $\varepsilon_{0}>0$ is a small constant to be determined later, and $q_0(x)$ satisfies
	\begin{equation} \label{13}
		q_0(x)\in C_{0}^{\infty}(\mathbb{R}^+),~~~q_0(x)\geq0,\;~~\int_{\mathbb{R}^{+}}q_0(x) {\rm d}x=1.
	\end{equation}	
	By solving \eqref{11}--\eqref{12}, we have
	\begin{equation}\label{hm}
		\hat m(x,t)=m_+{\rm e}^{-\alpha t}\displaystyle\int_{0}^{\varepsilon_{0}x}q_0(y){\rm d}y.
	\end{equation}
	Next, we turn to construct $\hat\rho(x,t)$. Integrating  \eqref{11}$_{1}$ over $(0,t)$ and using \eqref{hm} show
	\begin{align*}
		\int_{0}^{t}\hat\rho_{\tau}(x,\tau){\rm d}\tau&=-\int_{0}^{t}\hat m_x(x,\tau){\rm d}\tau\\
		&=-\int_{0}^{t}\varepsilon_{0}m_+q_0(\varepsilon_{0}x){\rm e}^{-\alpha\tau}{\rm d}\tau\\
		&=\frac{\varepsilon_{0}m_+}{\alpha}q_0(\varepsilon_{0}x){\rm e}^{-\alpha t}-\frac{\varepsilon_{0}m_+}{\alpha}q_0(\varepsilon_{0}x).
	\end{align*}
Thus $\hat\rho(x,t)$ can be constructed as
	\begin{equation}\notag
		\hat\rho(x,t)=\frac{\varepsilon_{0}m_+}{\alpha}q_0(\varepsilon_{0}x){\rm e}^{-\alpha t}.
	\end{equation}
With $\hat\rho(x,t)$, by \eqref{11}$_3$ and \eqref{12}$_2$, we obtain
	\begin{align}\notag
		\hat\phi(x,t)&=\hat\phi(x,0){\rm e}^{-bt}+\int_{0}^{t}{\rm e}^{-b(t-\tau)}a\hat\rho(x,\tau){\rm d}\tau\notag\\
		&=\hat\phi(x,0){\rm e}^{-bt}+\int_{0}^{t}\frac{\varepsilon_{0}am_+}{\alpha}q_0(\varepsilon_{0}x){\rm e}^{-bt}{\rm e}^{(b-\alpha)\tau}{\rm d}\tau.\notag
	\end{align}
	Obviously, we need to discuss whether $b$ and $\alpha$ are equal.
	If $b=\alpha$,
	\[
	\hat\phi(x,t)=\hat\phi(x,0){\rm e}^{-bt}+\frac{\varepsilon_{0}am_+}{\alpha}q_0(\varepsilon_{0}x)t{\rm e}^{-bt}.\]
	And if $b\neq\alpha$, 
	\[
	\hat\phi(x,t)=\hat\phi(x,0){\rm e}^{-bt}+\frac{\varepsilon_{0}am_+}{\alpha(b-\alpha)}q_0(\varepsilon_{0}x)\left({\rm e}^{-\alpha t}-{\rm e}^{-bt}\right).\]
	To sum up, we have
	\begin{equation}\notag
		\hat\phi(x,t)=\left\{\begin{array}{cl}
			d_+{\rm e}^{-bt}\displaystyle\int_{0}^{\varepsilon_{0}x}q(y){\rm d}y+\displaystyle\frac{\varepsilon_{0}am_+}{\alpha}q_0(\varepsilon_{0}x)t{\rm e}^{-bt},~~~~~~~~~~~~~~~~~ & \textrm{if}\ b=\alpha,\\[3mm]
			d_+{\rm e}^{-bt}\displaystyle\int_{0}^{\varepsilon_{0}x}q(y){\rm d}y+\displaystyle\frac{\varepsilon_{0}am_+}{\alpha(b-\alpha)}q_0(\varepsilon_{0}x)\left({\rm e}^{-\alpha t}-{\rm e}^{-bt}\right), & \textrm{if}\ b\neq\alpha.
		\end{array}\right.
	\end{equation}
	Thus, one can easily check that $(\hat\rho,\hat m,\hat\phi)(x,t)$ satisfies \eqref{11} and for $k, j \in \mathbb{N}$,
	\begin{equation}\label{13a}
		\partial^{k}_{x}\partial^{j}_{t}\hat\rho(0,t)=\partial^{k}_{x}\partial^{j}_{t}\hat m(0,t)=\partial^{k}_{x}\partial^{j}_{t}\hat\phi(0,t)=0.
	\end{equation}
	
	After constructing the correction function, the next step is to reformulate the IBVP \eqref{1}--\eqref{3a}. From \eqref{1}, \eqref{6} and \eqref{11}, we have
\begin{equation}\label{14}
		\left\{
		\begin{array}{lr}
			(\rho-\bar\rho-\hat\rho)_t+(m-\bar m-\hat m)_x=0, \\[2mm]
			(m-\bar m-\hat m)_t+\left(\frac{m^2}{\rho}\right)_x+(p(\rho)-p(\bar\rho))_x=\mu(\rho\phi_x-\bar\rho\bar{\phi}_x)-\alpha (m-\bar m-\hat m)-\bar{m}_t,\\[2mm]
			(\phi-\bar \phi-\hat \phi)_t=D\phi_{xx}+a(\rho-\bar\rho-\hat\rho)-b(\phi-\bar \phi-\hat \phi)-\bar{\phi}_t.
		\end{array}
		\right.
	\end{equation}
Integrating $\eqref{14}_{1}$ over $[0,+\infty) \times [0,t]$, and noticing that $\bar{\rho}(x,0)=\rho_{+}+\delta_{0}w_{0}(x)$ in Lemma \ref{l3.1}, we get
\begin{equation}\label{15}
\int_{0}^{+\infty}(\rho-\bar{\rho}-\hat{\rho})(y, t) {\rm d}y=\int_{0}^{+\infty}\left(\rho_{0}(x)-\rho_{+}\right) {\rm d} x-\delta_{0} \int_{0}^{+\infty} w_{0}(x) {\rm d} x-\frac{m_{+}}{\alpha}=0.
\end{equation}
Inspired by \cite{Hsiao-Liu1992}, let us define 
	\begin{equation}\label{15a}
		\left\{
		\begin{array}{lr}
			\varphi(x,t)=-\displaystyle\int^{+\infty}_{x}(\rho-\bar\rho-\hat\rho)(y,t){\rm d}y,\\[2mm]
			\psi(x,t)=(m-\bar m-\hat m)(x,t),\\[2mm]
			\zeta(x,t)=(\phi-\bar\phi-\hat\phi)(x,t),
		\end{array}
		\right.
	\end{equation}
then \eqref{14} can be transformed into
	\begin{equation}\label{18}
		\left\{
		\begin{array}{lr}
			\varphi_t+\psi=0,\\[2mm]
			\psi_t+\displaystyle\left(\frac{(\psi+\bar m+\hat m)^2}{\varphi_x+\bar\rho+\hat\rho}\right)_x+\left[p(\varphi_x+\bar\rho+\hat\rho)-p(\bar\rho)\right]_x\\[2mm]
			=\mu \varphi_x(\zeta_x+\bar\phi_x+\hat\phi_x)+\mu\bar\rho(\zeta_x+\hat\phi_x)+\mu\hat\rho(\zeta_x+\bar\phi_x+\hat\phi_x)-\alpha \psi-\bar m_t,\\[2mm]
			\zeta_t=D\zeta_{xx}+a\varphi_x-b\zeta-\bar\phi_t+D\bar\phi_{xx}+D\hat\phi_{xx},\\[2mm]
			\varphi(0,t)=\zeta_x(0,t)=0,\\[2mm]
			(\varphi,\psi,\zeta)(x,t)|_{t=0}=(\varphi_0,\psi_0,\zeta_0)(x)\to0,\;\;\;\mbox{as}\;\;\; x\to+\infty,
		\end{array}
		\right.
	\end{equation}
	or
	\begin{equation}\label{20}
		\left\{
		\begin{array}{lr}
			\varphi_{tt}-\left(p^{\prime}(\bar\rho)\varphi_x\right)_x	+\alpha \varphi_t\\[2mm]
			=-\mu \varphi_x\zeta_x-\mu \varphi_x\bar\phi_x-\mu \varphi_x\hat\phi_x-\mu\bar\rho\zeta_x-\mu\bar\rho\hat\phi_x-\mu\hat\rho\zeta_x-\mu\hat\rho\bar\phi_x-\mu\hat\rho\hat\phi_x-h_x-f_x,\\[2mm]
			\zeta_t=D\zeta_{xx}+a\varphi_x-b\zeta+g,\\[2mm]
			\varphi(0,t)=\zeta_x(0,t)=0,\\[2mm]
			\left(\varphi,\varphi_t,\zeta\right)(x,t)|_{t=0}=(\varphi_0,-\psi_0,\zeta_0)(x)\to0,\;\;\;\mbox{as}\;\;\; x\to+\infty,
		\end{array}
		\right.
	\end{equation}
	where
	\begin{equation}\label{22}
		\left\{
		\begin{array}{lr}
		\varphi_0(x):= -\displaystyle\int^{+\infty}_{x} \left[\rho_{0}(y)-\bar{\rho}_0(y)-\hat{\rho}(y,0) \right] {\rm d}y,\\[3mm]
\psi_0(x):=m_0(x)-\bar{m}(x,0)-\hat{m}(x,0),\\[3mm]
\zeta_0(x):=\phi_0(x)-\bar{\phi}(x,0)-\hat{\phi}(x,0),\\[2mm]
			h=-\displaystyle\frac{\left(\varphi_t+\frac{1}{\alpha}q(\bar\rho)_x-\hat m\right)^2}{\varphi_x+\bar\rho+\hat\rho},\\[2mm]
			f=\frac{1}{\alpha}q\left(\bar\rho\right)_t-\left[p\left(\varphi_x+\bar\rho+\hat\rho\right)-p\left(\bar\rho\right)-p^{\prime}\left(\bar\rho\right)\varphi_x\right],\\[2mm]
			g=-\bar\phi_t+D\bar\phi_{xx}+D\hat\phi_{xx}.
		\end{array}
		\right.
	\end{equation}
	Before stating our results precisely, we assume that 
	\begin{equation}\label{23}
		\left\{
		\begin{array}{lr}
			p(\rho)\in C^5(\mathbb{R}^{+}),\\[2mm]
			p'(\rho) - \frac{a\mu}{b}\rho > 0, \quad \text{for any} \quad \rho > 0.
		\end{array}
		\right.
	\end{equation}
	Now we state the convergence results as follows.
	\begin{theorem}\label{thm1}
	Suppose that \eqref{23} holds, $\left|\int_{0}^{\infty} (v_0(x) - v_+){\rm d}x - \frac{m_+}{\alpha}\right|$ and $\|\varphi_{0}\|_{3}+\|\psi_{0}\|_{2}+\|\zeta_{0}\|_{4}$ are sufficiency small.
		 Then, there exists a unique time-global solution $(\varphi,\psi,\zeta)(x,t)$ of \eqref{20}--\eqref{22}, which satisfies
		\begin{equation}\label{24}
			\left\{
			\begin{array}{lr}
				\|\partial_x^k\varphi(t)\|\leq C (1+t)^{-\frac{1}{2}(k+1)},\;\;\;k=0,1,2,\\[2mm]
				\|\partial_x^k\psi(t)\|\leq C (1+t)^{-\frac{1}{2}(k+2)},\;\;\;k=0,1,2,\\[2mm]
				\|\partial_x^k\zeta(t)\|\leq C (1+t)^{-\frac{1}{2}(k+1)},\;\;\;k=0,1.
			\end{array}
			\right.
		\end{equation}
	\end{theorem}
	\begin{remark}\label{remark32}
	It is worth pointing out that condition \eqref{23}$_2$ is indispensable. More precisely, if \eqref{23}$_2$ holds, then we can deduce that the following matrixe $P$ is positive definite:
		\[	
		P:=\left(\begin{array}{lr}
			p^{\prime}(\bar\rho) &-\mu\bar\rho \\
			-\mu\bar\rho & \frac{b\mu}{a}\bar\rho
		\end{array}\right),
		\]
		which implies for any $(x,y)$, we have
		\begin{equation}\label{positivedefine1}
			c\left(x^2+y^2\right)\leq p^{\prime}(\bar\rho)x^2-2\mu\bar\rho xy+\frac{b\mu}{a}\bar\rho y^2\leq C\left(x^2+y^2\right).
		\end{equation}
		With \eqref{positivedefine1}, we can deal with some troublesome terms. See \eqref{13.43}, \eqref{13.62} for more details.
	\end{remark}
	\begin{remark}\label{R1}
As one can see from \eqref{12}, we introduce a small positive constant $\varepsilon_0$. When we assume that $\varepsilon_0$ is small enough instead of $|u_+|$, we can still get the expected estimates about $(\hat{v},\hat{u})$ (see \eqref{3.25}). So we do not need to assume $\|v_{0}-v_{+}\|_{L^{1}}+|u_{+}|\ll 1$. This method of constructing the correction function can be regarded as a generalization of that used in \cite{Hsiao-Liu1992, Nishihara-Yang1999}.
\end{remark}
	Noting that $\rho-\bar{\rho}=\varphi_{x}+\hat{\rho},~m-\bar{m}=z+\hat{m},~{\rm and}~\phi-\bar{\phi}=\zeta+\hat{\phi}$, and using \eqref{24} and \eqref{3.25}, we can immediately obtain the following corollary.
	\begin{corollary}[{\rm Converge to diffusion waves}]
		Suppose that all the conditions in
		Theorem \ref{thm1} hold, the system \eqref{1}--\eqref{3a} possesses a unique global solution $(\rho,m,\phi)(x,t)$ satisfying
		\begin{equation}\notag
			\left\{
			\begin{array}{lr}
				\|(\rho-\bar\rho)(t)\|_{L^{\infty}}\leq C (1+t)^{-\frac{3}{4}},\\[2mm]
				\|\left(m-\bar m\right)(t)\|_{L^{\infty}}\leq C (1+t)^{-\frac{5}{4}},\\[2mm]
				\|\left(\phi-\bar\phi\right)(t)\|_{L^{\infty}}\leq C (1+t)^{-\frac{3}{4}}.	
			\end{array}
			\right.
		\end{equation}
	\end{corollary}

\subsection{Initial-boundary value problem with boundary condition \eqref{3b}}

In this case, notice that 
\begin{equation}\notag
	\frac{\rm d}{{\rm d} t}\rho(0,t)=-m(0,t)=0, ~~\mbox{and then}~~\rho(0,t)=\rho_0(0).
\end{equation}
So we study it in two situations.

\subsubsection{The case of $\rho_0(0)\neq \rho_{+}$}

From \cite{Duyn-Peletier1977,Nishihara-Yang1999}, we can confirm that for any positive constant $\rho_{0}(0)$ and $\rho_{+}$, there exists a unique solution $\bar{\rho}(x, t)$ in the form of $\bar{\rho}(\frac{x}{1+t})|_{x \geq 0}$ satisfying
\begin{equation}\label{1.2}
\left\{\begin{array}{l}
\bar\rho_t-\frac{1}{\alpha}q(\bar\rho)_{xx}=0,\\[2mm]
\bar{\rho}|_{x=0}=\rho_{0}(0),\\[2mm]
 \bar{\rho}|_{x=+\infty}=\rho_{+},\quad (x,t)\in \mathbb{R}^{+}\times\mathbb{R}^{+},.
\end{array}
        \right.	
\end{equation}
 Moreover, $(\bar{m},\bar\phi)(x,t)$ is defined by
\begin{equation}\label{1.2a}
\bar{m}(x,t)=-\frac{1}{\alpha}q(\bar{\rho})_{x}, \ \ \bar\phi=\frac{a}{b}\bar\rho.	
\end{equation}
then
\begin{equation}\notag
\bar{m}_{x}|_{x=0}=\bar{\rho}_{t}|_{x=0}=\left.\bar{\rho}^{\prime}\left(\frac{x}{\sqrt{t+1}}\right)\left(-\frac{x}{2 \sqrt{t+1}(t+1)}\right)\right|_{x=0}=0.	
\end{equation}
Thus,  $(\bar{v},\bar{u})(x, t)$ called the nonlinear diffusion wave which satisfies \eqref{6} and \eqref{8}.

Next, similar to the previous idea, we still need to construct a set of correction functions. Inspired by \cite{Nishihara-Yang1999}, we define the correction functions as
\begin{equation}\label{3.8}
 \left\{ \begin{array}{l}
 \hat{\rho}(x,t)=\displaystyle\frac{\varepsilon_{0}(m_{0}(0)-m_{+})q_0(\varepsilon_{0}x)}{-\alpha}{\rm e}^{-\alpha t},\\[3pt]
\hat{m}(x,t) = {\rm e}^{-\alpha t} \left(m_+ + (m_{0}(0)-m_{+}) \displaystyle\int^{+\infty}_{\varepsilon_0 x} q_0(y){\rm d}y\right),
\end{array} \right.
\end{equation}
and 
\begin{equation}\label{3.8a}
		\hat\phi(x,t)=\left\{\begin{array}{cl}
			d_+{\rm e}^{-bt}\displaystyle\int_{0}^{\varepsilon_0 x} q_0(y){\rm d}y-\displaystyle\frac{\varepsilon_{0}a(m_{0}(0)-m_{+})}{\alpha}q_0(\varepsilon_{0}x)t{\rm e}^{-bt},~~~~~~~~~~~~ & \textrm{if}\ b=\alpha,\\[3mm]
			d_+{\rm e}^{-bt}\displaystyle\int_{0}^{\varepsilon_0 x} q_0(y){\rm d}y-\displaystyle\frac{\varepsilon_{0}a(m_{0}(0)-m_{+})}{\alpha(b-\alpha)}q_0(\varepsilon_{0}x)\left({\rm e}^{-\alpha t}-{\rm e}^{-bt}\right), & \textrm{if}\ b\neq\alpha,
		\end{array}\right.
	\end{equation}
where $\varepsilon_{0}>0$ is a small constant to be specified later, and $q_{0}(x)$ satisfies \eqref{13}.
 Therefore, we have
\begin{equation}\label{3.9}
 \left\{ \begin{array}{l}
  \hat\rho_t+\hat m_x=0,     \\[2mm]
			\hat m_t=-\alpha\hat m, \\[2mm]
			\hat\phi_t=a\hat\rho-b\hat\phi, \\[2mm]
  (\partial^{k}_{x}\partial^{j}_{t}\hat{\rho},\hat{m}, \partial^{k}_{x}\partial^{j}_{t}\hat{m}_x, \partial^{k}_{x}\partial^{j}_{t}\hat\phi)|_{x=0}=(0, m_{0}(0){\rm e}^{-\alpha t} ,0,0), \\[2mm]
   (\hat{\rho},\hat{m},\hat\phi)|_{x=+\infty}=(0, m_{+}{\rm e}^{-\alpha t} ,d_+{\rm e}^{-bt}).
  \end{array} \right.
\end{equation}
Let us define the
	perturbation function $(\varphi,\psi,\zeta)(x,t)$ as
	\begin{equation}\notag
		\left\{
		\begin{array}{lr}
			\varphi(x,t)=-\displaystyle\int^{+\infty}_{x}(\rho-\bar\rho-\hat\rho)(y,t){\rm d}y,\\[2mm]
			\psi(x,t)=(m-\bar m-\hat m)(x,t),\\[2mm]
			\zeta(x,t)=(\phi-\bar\phi-\hat\phi)(x,t).
		\end{array}
		\right.
	\end{equation}
	Then we can deduce 
	\begin{equation}\label{3.10}
		\left\{
		\begin{array}{lr}
			\varphi_t+\psi=0,\\[2mm]
			\psi_t+\displaystyle\left(\frac{(\psi+\bar m+\hat m)^2}{\varphi_x+\bar\rho+\hat\rho}\right)_x+\left[p(\varphi_x+\bar\rho+\hat\rho)-p(\bar\rho)\right]_x\\[2mm]
			=\mu \varphi_x(\zeta_x+\bar\phi_x+\hat\phi_x)+\mu\bar\rho(\zeta_x+\hat\phi_x)+\mu\hat\rho(\zeta_x+\bar\phi_x+\hat\phi_x)-\alpha \psi-\bar m_t,\\[2mm]
			\zeta_t=D\zeta_{xx}+a\varphi_x-b\zeta-\bar\phi_t+D\bar\phi_{xx}+D\hat\phi_{xx},\\[2mm]
			\varphi_x(0,t)=\zeta(0,t)=0,\\[2mm]
			(\varphi,\psi,\zeta)(x,t)|_{t=0}=(\varphi_0,\psi_0,\zeta_0)(x)\to0,\;\;\;\mbox{as}\;\;\; x\to+\infty.
		\end{array}
		\right.
	\end{equation}
	or
	\begin{equation}\label{3.11}
		\left\{
		\begin{array}{lr}
			\varphi_{tt}-\left(p^{\prime}(\bar\rho)\varphi_x\right)_x	+\alpha \varphi_t\\[2mm]
			=-\mu \varphi_x\zeta_x-\mu \varphi_x\bar\phi_x-\mu \varphi_x\hat\phi_x-\mu\bar\rho\zeta_x-\mu\bar\rho\hat\phi_x-\mu\hat\rho\zeta_x-\mu\hat\rho\bar\phi_x-\mu\hat\rho\hat\phi_x-h_x-f_x,\\[2mm]
			\zeta_t=D\zeta_{xx}+a\varphi_x-b\zeta+g,\\[2mm]
			\varphi_x(0,t)=\zeta(0,t)=0,\\[2mm]
			\left(\varphi,\varphi_t,\zeta\right)(x,t)|_{t=0}=(\varphi_0,-\psi_0,\zeta_0)(x)\to0,\;\;\;\mbox{as}\;\;\; x\to+\infty,
		\end{array}
		\right.
	\end{equation}
	where
	\begin{equation}\label{3.12}
		\left\{
		\begin{array}{lr}
		\varphi_0(x):= -\displaystyle\int^{+\infty}_{x} \left[\rho_{0}(y)-\bar{\rho}_0(y)-\hat{\rho}(y,0) \right] {\rm d}y,\\[3mm]
\psi_0(x):=m_0(x)-\bar{m}(x,0)-\hat{m}(x,0),\\[3mm]
\zeta_0(x):=\phi_0(x)-\bar{\phi}(x,0)-\hat{\phi}(x,0),\\[2mm]
			h=-\displaystyle\frac{\left(\varphi_t+\frac{1}{\alpha}q(\bar\rho)_x-\hat m\right)^2}{\varphi_x+\bar\rho+\hat\rho},\\[2mm]
			f=\frac{1}{\alpha}q\left(\bar\rho\right)_t-\left[p\left(\varphi_x+\bar\rho+\hat\rho\right)-p\left(\bar\rho\right)-p^{\prime}\left(\bar\rho\right)\varphi_x\right],\\[2mm]
			g=-\bar\phi_t+D\bar\phi_{xx}+D\hat\phi_{xx}.
		\end{array}
		\right.
	\end{equation}

	\begin{theorem}\label{thm2}
		Suppose that \eqref{23} holds, $\left|\rho_0(0) - \rho_+\right|$ and $\|\varphi_{0}\|_{3}+\|\psi_{0}\|_{2}+\|\zeta_{0}\|_{4}$ are sufficiency small.
		then the global solution $(\varphi,\psi,\zeta)(x,t)$ of \eqref{3.10} uniquely exists and satisfies the following decay rates:
		\begin{equation}\label{3.14}
			\left\{
			\begin{array}{lr}
				\|\partial_x^k\varphi(t)\|\leq C (1+t)^{-\frac{1}{2}(k+1)},\;\;\;k=0,1,2,\\[2mm]
				\|\partial_x^k\psi(t)\|\leq C (1+t)^{-\frac{1}{2}(k+2)},\;\;\;k=0,1,2,\\[2mm]
				\|\partial_x^k\zeta(t)\|\leq C (1+t)^{-\frac{1}{2}(k+1)},\;\;\;k=0,1.
			\end{array}
			\right.
		\end{equation}
	\end{theorem}
	With Theorem \ref{thm2} in hand, we can deduce the following corollary.	
\begin{corollary}[{\rm Converge to diffusion waves}]
		Suppose that all the conditions in
		Theorem \ref{thm2} hold, the system \eqref{1}--\eqref{3} and \eqref{3b}  possesses a unique global solution $(\rho,m,\phi)(x,t)$ satisfying
		\begin{equation}\notag
			\left\{
			\begin{array}{lr}
				\|(\rho-\bar\rho)(t)\|_{L^{\infty}}\leq C (1+t)^{-\frac{3}{4}},\\[2mm]
				\|\left(m-\bar m\right)(t)\|_{L^{\infty}}\leq C (1+t)^{-\frac{5}{4}},\\[2mm]
				\|\left(\phi-\bar\phi\right)(t)\|_{L^{\infty}}\leq C (1+t)^{-\frac{3}{4}}.	
			\end{array}
			\right.
		\end{equation}
	\end{corollary}
\subsubsection{The case of $\rho_0(0)= \rho_{+}$}

In this case, we take
\begin{equation}\notag
(\bar{\rho},\bar{m}, \bar\phi)(x,t)\equiv (\rho_{+},0,\frac{a}{b}\rho_{+}),
\end{equation}
 and the correction function is the same as that in \eqref{3.8}--\eqref{3.8a}.
Then define the
	perturbation as
	\begin{equation}\notag
		\left\{
		\begin{array}{lr}
			\varphi(x,t)=-\displaystyle\int^{+\infty}_{x}(\rho-\rho_{+}-\hat\rho)(y,t){\rm d}y,\\[2mm]
			\psi(x,t)=(m-\hat m)(x,t),\\[2mm]
			\zeta(x,t)=(\phi-\frac{a}{b}\rho_{+}-\hat\phi)(x,t).
		\end{array}
		\right.
	\end{equation}
Similar to \eqref{3.10}, we can deduce the following reformulated problem 
	\begin{equation}\label{n3}
		\left\{
		\begin{array}{lr}
			\varphi_t+\psi=0,\\[2mm]
			\psi_t+\displaystyle\left(\frac{(\psi+\hat m)^2}{\varphi_x+\rho_{+}+\hat\rho}\right)_x+\left[p(\varphi_x+\rho_{+}+\hat\rho)-p(\rho_{+})\right]_x\\[2mm]
			=\mu \varphi_x(\zeta_x+\hat\phi_x)+\mu\rho_{+}(\zeta_x+\hat\phi_x)+\mu\hat\rho(\zeta_x+\hat\phi_x)-\alpha \psi-\bar m_t,\\[2mm]
			\zeta_t=D\zeta_{xx}+a\varphi_x-b\zeta+D\hat\phi_{xx},\\[2mm]
			\varphi_x(0,t)=\zeta(0,t)=0,\\[2mm]
			(\varphi,\psi,\zeta)(x,t)|_{t=0}=(\varphi_0,\psi_0,\zeta_0)(x)\to0,\;\;\;\mbox{as}\;\;\; x\to+\infty.
		\end{array}
		\right.
	\end{equation}

	\begin{theorem}\label{thm3}
		Suppose that \eqref{23} holds, $\|\varphi_{0}\|_{3}+\|\psi_{0}\|_{2}+\|\zeta_{0}\|_{4}$ is sufficiency small.
		then the global solution $(\varphi,\psi,\zeta)(x,t)$ of \eqref{n3} uniquely exists and satisfies the following decay rates:
		\begin{equation}\notag
			\left\{
			\begin{array}{lr}
				\|\partial_x^k\varphi(t)\|\leq C (1+t)^{-\frac{1}{2}(k+1)},\;\;\;k=0,1,2,\\[2mm]
				\|\partial_x^k\psi(t)\|\leq C (1+t)^{-\frac{1}{2}(k+2)},\;\;\;k=0,1,2,\\[2mm]
				\|\partial_x^k\zeta(t)\|\leq C (1+t)^{-\frac{1}{2}(k+1)},\;\;\;k=0,1.
			\end{array}
			\right.
		\end{equation}
	\end{theorem}
	Therefore, we have the following corollary.	
\begin{corollary}[{\rm Converge to constant states}]
		Suppose that all the conditions in
		Theorem \ref{thm3} hold, the system \eqref{1}--\eqref{3} and \eqref{3b}  possesses a unique global solution $(\rho,m,\phi)(x,t)$ satisfying
		\begin{equation}\label{22b}
			\left\{
			\begin{array}{lr}
				\|(\rho-\rho_+)(t)\|_{L^{\infty}}\leq C (1+t)^{-\frac{3}{4}},\\[2mm]
				\| m(t)\|_{L^{\infty}}\leq C (1+t)^{-\frac{5}{4}},\\[2mm]
				\|\left(\phi-\frac{a}{b}\rho_{+}\right)(t)\|_{L^{\infty}}\leq C (1+t)^{-\frac{3}{4}}.	
			\end{array}
			\right.
		\end{equation}
	\end{corollary}

\section{Proof of Theorem \ref{thm1}}\label{s3}
\subsection{Preliminaries}

In this subsection, we introduce some results on some fundamental properties of the diffusion wave $(\bar{\rho},\bar{m}, \bar\phi)(x,t)$ and the correction function $(\hat{\rho},\hat{m},\hat\phi)(x,t)$, which will be used frequently later.
Firstly, as for the solution $(\bar{\rho},\bar{m}, \bar\phi)(x,t)$ to IBVP \eqref{6}--\eqref{7}, we can obtain the following lemma.
\begin{lemma}\label{l3.1}
 For IBVP \eqref{6}$_1$--\eqref{7}, we define $\bar{\rho}(x,0)=\rho_{+}+\delta_{0}w_0(x)$ as the initial value of this system, where $w_{0}$ is a given sufficiently smooth function satisfying  $w_{0}\in L^{1}(\mathbb{R}^{+})~ \mbox{and} ~\int_{0}^{+\infty}w_{0} {\rm d}x\neq0$, and $\delta_{0}$ is a constant satisfying \eqref{15}.
 Then for $k,j\geq 0$, the solution $(\bar{\rho},\bar{m}, \bar\phi)(x,t)$ to IBVP \eqref{6}--\eqref{7}  is 
 globally exists, and satisfies
\begin{equation}\label{3.15}
 \left\{\begin{array}{l}
\|\partial^{k}_{x}\partial^{j}_{t}(\bar{\rho}-\rho_{+})(t)\|\leq C|\delta_{0}|(1+t)^{-\frac{k}{2}-j},\\[2mm]
\|\partial^{k}_{x}\partial^{j}_{t}(\bar{\phi}-\frac{a}{b}\rho_{+})(t)\|\leq C|\delta_{0}|(1+t)^{-\frac{k}{2}-j},\\[2mm]
\|\partial^{k}_{x}\partial^{j}_{t}\bar{m}(t)\|\leq C|\delta_{0}|(1+t)^{-\frac{1}{2}-\frac{k}{2}-j}.
  \end{array}\right.
\end{equation}
\end{lemma}
\begin{proof}
Inspired by \cite{Mei2005}, we construct the diffusion wave $\bar\rho(x,t)$ as
\begin{equation}\label{3.16}
\bar\rho(x,t)=\rho_++\delta_{0}w(x,t).
\end{equation}
Then from \eqref{6}, we have 
 \begin{equation}\label{3.17}
 	\left\{
		\begin{array}{lr}
			\delta_{0}w_t-\displaystyle\frac{1}{\alpha}(q'(\rho_++\delta_{0}w_x)\delta_{0}w)_{x}=0,     \\[2mm]
			w_x(0,t)=0,\ \ w(+\infty,t)=0, \\[2mm]
			w(x,0)=w_0(x),
		\end{array}
		\right.
 \end{equation}
 or 
 \begin{equation}\label{3.18}
 	\left\{
		\begin{array}{lr}
			\bar w_t-\displaystyle\frac{1}{\alpha}(q'(\rho_++\bar w)\bar w_x)_{x}=0,     \\[2mm]
			\bar w_x(0,t)=0,\ \ \bar w(+\infty,t)=0, \\[2mm]
			\bar w(x,0):=\bar w_0(x)=\delta_{0}w_0(x).
		\end{array}
		\right.
 \end{equation}
 Now we will use the energy method to prove the existence and uniqueness of the solution of \eqref{3.17} or \eqref{3.18}.
 The local existence of the solution is standard, then to get the global existence,
we only prove the following a {\it priori} estimates.
For $T>0$, we denote
\begin{align}\label{3.19}
N(T):= \sup \limits_{0 \leq t \leq T}\sum_{k,j\in \mathbb{N}}(1+t)^{k+2j}\|\partial_{x}^{k}\partial_{t}^{j} \bar w(t)\|^{2}\leq \varepsilon^{2}\ll1,
\end{align}
then for $k,j\in \mathbb{N}^+$, we prove that
\begin{equation}\label{3.19a}
\|\partial_x^k\partial_t^j\bar w_{t}\|\lesssim\|\bar w_{0}\|_{k+2j}(1+t)^{-\frac{k}{2}-j},	
\end{equation}
which implies \eqref{3.15}$_1$.
Multiplying \eqref{3.18}$_1$ by $\bar w$ and integrating it with respect to $x$ over $\mathbb{R}^+$ give
\begin{equation}\notag
 \frac{{\rm d}}{{\rm d}t}\int_{\mathbb{R}^+} \bar{w}^2{\rm d} x+\frac{1}{\alpha}\int_{\mathbb{R}^+} q'(\rho_++\bar w)\bar w_{x}^{2}{\rm d}x =0.
\end{equation}
Noticing that $q'(\rho_++\bar w)>c>0$, we get
\begin{equation}\label{3.20}
   \|\bar w(t)\|^{2}+\int_{0}^{t}\|\bar w_{x}(\tau)\|^{2}{\rm d}\tau \lesssim  \|\bar w_{0}\|^{2}.
  \end{equation}
Then multiplying \eqref{3.18}$_1$ by $-{\bar w}_{xx}$ and integrating it with respect to $x$ over $\mathbb{R}^+$, we have
\begin{align}\label{3.20aa}
 \frac{{\rm d}}{{\rm d}t}\int_{\mathbb{R}^+} \bar{w}_x^2{\rm d} x+\frac{1}{\alpha}\int_{\mathbb{R}^+} q'(\rho_++\bar w)\bar w_{xx}^{2}{\rm d}x &=-\frac{1}{\alpha}\int_{\mathbb{R}^+} q''(\rho_++\bar w)\bar w_{x}^2\bar w_{xx}{\rm d}x\notag\\
 &\leq C\|\bar w_{x}\|_{L^\infty}\|\bar w_{x}\|\|\bar w_{xx}\|\notag\\
 &\leq \eta\|\bar w_{xx}\|^2+C_\eta\varepsilon^2(1+t)^{-\frac{3}{2}}\|\bar w_{x}\|^2.
\end{align}
Integrating the above inequality over $[0, t]$ and using \eqref{3.20}, we obtain
\begin{equation}\label{3.21}
   \|\bar w_x(t)\|^{2}+\int_{0}^{t}\|\bar w_{xx}(\tau)\|^{2}{\rm d}\tau \lesssim  \|\bar w_{0}\|_1^{2}.
  \end{equation}
  Then multiplying \eqref{3.20aa} by $(1 + t)$, integrating it with respect to $t$ and using \eqref{3.20} and \eqref{3.21} lead to
\begin{equation}\label{3.22}
  (1+t) \|\bar w_x(t)\|^{2}+\int_{0}^{t}(1+\tau)\|\bar w_{xx}(\tau)\|^{2}{\rm d}\tau \lesssim  \|\bar w_{0}\|_1^{2}.
  \end{equation}
Similarly, we can calculate
\begin{align}\notag
 \frac{{\rm d}}{{\rm d}t}\int_{\mathbb{R}^+} \bar{w}_{xx}^2{\rm d} x+\frac{1}{\alpha}\int_{\mathbb{R}^+} q'(\rho_++\bar w)\bar w_{xxx}^{2}{\rm d}x &\leq C\int_{\mathbb{R}^+} (|\bar w_x|^3+|\bar w_x\bar w_{xx}|)|\bar w_{xxx}|{\rm d}x\notag\\
 &\leq C(\|\bar w_{x}\|^2_{L^\infty}\|\bar w_{x}\|+\|\bar w_{x}\|_{L^\infty}\|\bar w_{xx}\|)\|\bar w_{xxx}\|\notag\\
 &\leq \eta\|\bar w_{xxx}\|^2+C_\eta\varepsilon^2((1+t)^{-3}\|\bar w_{x}\|^2+(1+t)^{-\frac{3}{2}}\|\bar w_{xx}\|^2).\notag
\end{align}
Then we can prove
\begin{equation}\label{3.23}
  (1+t)^2 \|\bar w_{xx}(t)\|^{2}+\int_{0}^{t}(1+\tau)^2\|\bar w_{xxx}(\tau)\|^{2}{\rm d}\tau \lesssim  \|\bar w_{0}\|_2^{2}.
  \end{equation}
More generally, for $k\in \mathbb{N}$, we can get 
\begin{equation}\label{3.24}
  (1+t)^k \|\partial_x^k\bar w(t)\|^{2}\lesssim  \|\bar w_{0}\|_k^{2}.
  \end{equation} 
For $\|\partial_x^k\bar w_{t}\|$, by applying \eqref{3.18}$_1$, we can estimate
\begin{equation}\notag
\|\partial_x^k\bar w_{t}\|\lesssim\|\bar w_{0}\|_{k+2}(1+t)^{-\frac{k}{2}-1}.	
\end{equation}
To sum up, we can prove \eqref{3.19a}. Then notice that $\bar m=-\frac{1}{\alpha}q(\bar\rho)_x, \bar\phi=\frac{a}{b}\bar\rho,$ we can also get \eqref{3.15}$_2$ and \eqref{3.15}$_3$.
\end{proof}
\begin{remark}
	 It should be noted that from \eqref{15}, we can immediately obtain
	 \begin{equation}\label{3.25a}
	|\delta_0|\leq C\left|\int_{0}^{+\infty} (v_0(x) - v_+){\rm d}x - \frac{m_+}{\alpha}\right|. 	
	 \end{equation}
\end{remark}

Next, as shown in \eqref{11}, the correction function  $(\hat{\rho},\hat{m},\hat\phi)(x,t)$ satisfies the following estimates.
\begin{lemma}\label{l3.2}
 For $p \in [1, \infty]$, it holds that for $k, j \geq 0$,
 \begin{equation}\label{3.25}
 \left\{\begin{array}{l}
    \left\|\partial_{x}^{k} \partial_{t}^{j} \hat{\rho}(t)\right\|_{L^{p}}\lesssim\varepsilon_{0}^{k+1-\frac{1}{p}} {\rm e}^{-ct}, \\[2mm]
    \left\|\partial_{x}^{k} \partial_{t}^{j} [\hat{m}_x,\hat{\phi}_x](t)\right\|_{L^{p}}\lesssim\varepsilon_{0}^{k+1-\frac{1}{p}}{\rm e}^{-ct}, \\[2mm]
    \left\|\partial_{t}^{j}[\hat{m},\hat{\phi}](t)\right\|_{L^{\infty}}\lesssim{\rm e}^{-ct}.
\end{array}\right.
\end{equation}
\end{lemma}

\subsection{Global existence}\label{s3.2}
In this subsubsection, we aim to study global existence and uniqueness of solution to \eqref{20}--\eqref{22} in the $L^2$-framework. 
It is well known that the global existence
can be obtained by the continuation argument based on the local existence of solution and a {\it priori}
estimates. The local existence of the solution for \eqref{20}--\eqref{22} can be proved by the standard iteration
method (cf.\cite{Nishida1978,ma1977}). Next, we devote ourselves to establishing the following a {\it priori} estimates.
	\begin{proposition}\label{p3.1}
		Under the condition of Theorem \ref{thm1},  and further assume that $\delta:=\left|\delta_{0}\right|+\varepsilon_0^{\frac{1}{2}}$ and $\|\varphi_{0}\|_{3}+\|\psi_{0}\|_{2}+\|\zeta_{0}\|_{4}$ are sufficiency small. 
		Then for any given $T>0$, there are positive constants $\varepsilon$ and $C$ such that if the solution $(\varphi, \zeta)(x, t)$ to \eqref{20}--\eqref{22} on $0\leq t\leq T$ satisfies
		\begin{align}\label{13.15}
		&\sup_{0\leq t\leq T}\left\{\sum_{k=0}^{2}(1+t)^{k}\|\partial_x^k\varphi(t)\|^2+\sum_{k=0}^{2}(1+t)^{k+2}\|\partial_x^k\varphi_t(t)\|^2+\sum_{k=0}^{1}(1+t)^{k+1}\|\partial_x^k[\zeta,\zeta_x](t)\|^2\notag\right.\\ 
		&\ \ \ \ \ \ \ \ \ \ \ \ \ \ \ \ \ \ \ \ \ \ \ \ \ \ \ \ \left.+\sum_{k=0}^{1}(1+t)^{k+3}\|\partial_x^k[\zeta_t,\zeta_{xt}](t)\|^2+(1+t)^{2}\|(\partial_x^3\varphi,\partial_x^3\zeta)(t)\|^2\right\}\leq \varepsilon^2,
		\end{align}
		then we have
		\begin{align}\label{13.16}
			&\;\;\;\;\sum_{k=0}^{2}(1+t)^{k}\|\partial_x^k\varphi(t)\|^2+\sum_{k=0}^{1}(1+t)^{k+1}\|\partial_x^k[\zeta,\zeta_x](t)\|^2+\sum_{k=0}^{1}\int_{0}^{t}(1+\tau)^k\|\partial_x^k[\varphi_x,\zeta,\zeta_x](\tau)\|^2{\rm d}\tau\notag\\
			&\leq C(\|\varphi_0\|_3^2+\|\psi_0\|_2^2+\|\zeta_0\|_4^2+\delta),
		\end{align}
		and
		\begin{align}\label{13.17}
			&\;\;\;\;\sum_{k=0}^{2}(1+t)^{k+2}\|\partial_x^k\varphi_t(t)\|^2+\sum_{k=0}^{1}(1+t)^{k+3}\|\partial_x^k[\zeta_t,\zeta_{xt}](t)\|^2+(1+t)^{2}\|(\partial_x^3\varphi,\partial_x^3\zeta)(t)\|^2\notag\\
			&\;\;\;\;+\sum_{k=0}^{2}\int_{0}^{t}(1+\tau)^{k+1}\|\partial_x^k\varphi_{\tau}(\tau)\|^2{\rm d}\tau+\sum_{k=0}^{1}\int_{0}^{t}(1+\tau)^{k+2}\|\partial_x^k[\zeta_{\tau},\zeta_{x\tau}](\tau)\|^2{\rm d}\tau\notag\\
			&\leq C(\|\varphi_0\|_3^2+\|\psi_0\|_2^2+\|\zeta_0\|_4^2+\delta).
		\end{align}
	\end{proposition}
	By applying \eqref{13.15}, one gets
	\begin{align}\label{13.18}
		\left\{
		\begin{array}{lr}
		\|\partial_x^k\varphi(t)\|_{L^{\infty}}\lesssim\varepsilon(1+t)^{-\frac{1}{4}-\frac{k}{2}} (k=0,1),\ \ \ \|\partial_x^k\varphi_t(t)\|_{L^{\infty}}\lesssim\varepsilon(1+t)^{-\frac{5}{4}-\frac{k}{2}}(k=0,1),\\[2mm]
			\|\partial_t^j\zeta(t)\|_{L^{\infty}}\lesssim\varepsilon(1+t)^{-\frac{3}{4}-j} (j=0,1),\ \ \ \|\partial_x^k\zeta_x(t)\|_{L^{\infty}}\lesssim\varepsilon(1+t)^{-1} (k=0,1),\\[2mm]
			\|\zeta_{xt}(t)\|_{L^{\infty}}\lesssim\varepsilon(1+t)^{-2},\ \  \|\varphi_{xx}(t)\|_{L^{\infty}}\lesssim\varepsilon(1+t)^{-1}. 
		\end{array}
		\right.
	\end{align}
	which will be used later. 
	In addition, from \eqref{7}, \eqref{13a} and \eqref{18}, one can immediately obtain the following boundary conditions
	\begin{equation}\label{13.19}
		\varphi(0,t)=\varphi_{xx}(0,t)=\zeta_x(0,t)=0, ~\mbox{etc}.
	\end{equation}
	Now, we are ready to prove Proposition \ref{p3.1}, which will be given by the following several lemmas.
	\begin{lemma}\label{l3.3}
	Under the hypothesis of Proposition \ref{p3.1}, it holds that
	\begin{align}\label{43}
			&\|\varphi(t)\|^2+(1+t)\|[\varphi_x,\varphi_t,\zeta,\zeta_x](t)\|^2+\int_{0}^{t}\left\{\|[\varphi_x,\zeta,\zeta_x](\tau)\|^2+(1+\tau)\|[\varphi_{\tau},\zeta_{\tau}](\tau)\|^2\right\}{\rm d}\tau\notag\\ 
			&\leq C(\|\varphi_0\|_1^2+\|\psi_0\|^2+\|\zeta_0\|_1^2+\delta)+C\int_{0}^{t}{\rm e}^{-c\tau}\left(\|\varphi(\tau)\|_2^2+\|\varphi_{\tau}(\tau)\|_1^2\right){\rm d}\tau.
		\end{align}
	\end{lemma}
	\begin{proof}
		By integration by parts, we can rewrite \eqref{20} as 
		\begin{equation}\label{29}
			\left\{
			\begin{array}{lr}
				\varphi_{tt}-(p^{\prime}(\bar\rho)\varphi_x)_x+\alpha \varphi_t+\mu(\bar\rho\zeta)_x\\[2mm]
				=-\mu \varphi_x\zeta_x-\mu \varphi_x\bar\phi_x-\mu \varphi_x\hat\phi_x+\mu\bar\rho_x\zeta-\mu\bar\rho\hat\phi_x-\mu\hat\rho\zeta_x-\mu\hat\rho\bar\phi_x-\mu\hat\rho\hat\phi_x-h_x-f_x,\\[2mm]
				\zeta_t-D\zeta_{xx}-a\varphi_x+b\zeta=g.
			\end{array}
			\right.
		\end{equation}
		By calculating $\int_{\mathbb{R}^+}\left(\eqref{29}_1\times\varphi\right){\rm d}x$, and using $\varphi_{tt}\varphi=(\varphi_t\varphi)_t-\varphi_t^2$, $\bar\phi=\frac{a}{b}\bar\rho$ and \eqref{13.19},  we arrive at
\begin{align}\label{31}
			&\;\;\;\;\;	\frac{{\rm d}}{{\rm d}t}\left\{\frac{\alpha}{2}\int_{\mathbb{R}^+}\varphi^2{\rm d}x+\int_{\mathbb{R}^+}^{}\varphi_t\varphi {\rm d}x\right\}+\int_{\mathbb{R}^+}^{}p^{\prime}(\bar\rho)\varphi_x^2 {\rm d}x-\mu\int_{\mathbb{R}^+}^{}\bar\rho\zeta\varphi_x {\rm d}x\notag\\
			&=\int_{\mathbb{R}^+}\varphi_t^2{\rm d}x-\mu\int_{\mathbb{R}^+}^{}\varphi_x\zeta_x\varphi {\rm d}x-\frac{\mu}{b}\int_{\mathbb{R}^+}^{}[\bar\rho_x(a\varphi_x-b\zeta)]\varphi{\rm d}x-\mu\int_{\mathbb{R}^+}^{}\varphi_x\hat\phi_x\varphi{\rm d}x\notag\\
			&\;\;\;\;-\mu\int_{\mathbb{R}^+}^{}\hat\rho\zeta_x\varphi{\rm d}x	-\mu\int_{\mathbb{R}^+}^{}\bar\rho\hat\phi_x\varphi{\rm d}x-\mu\int_{\mathbb{R}^+}\hat\rho\bar\phi_x\varphi{\rm d}x	-\mu\int_{\mathbb{R}^+}\hat\rho\hat\phi_x\varphi{\rm d}x\notag\\
			&\;\;\;\;
		-\int_{\mathbb{R}^+}^{}f_x\varphi{\rm d}x-\int_{\mathbb{R}^+}h_x\varphi{\rm d}x.
		\end{align}
		Then by calculating $\int_{\mathbb{R}^+}\left(\eqref{29}_2\times\left(\frac{\mu}{a}\bar\rho\zeta\right)\right){\rm d}x$ and using \eqref{13.19}, one has
		\begin{align}\label{32}
			&\;\;\;\;\frac{\mu}{2a}	\frac{{\rm d}}{{\rm d}t}\int_{\mathbb{R}^+}^{}\bar\rho\zeta^2{\rm d}x+\frac{\mu D}{a}\int_{\mathbb{R}^+}^{}\bar\rho\zeta_x^2{\rm d}x+\frac{b\mu}{a}\int_{\mathbb{R}^+}^{}\bar\rho\zeta^2{\rm d}x-\mu\int_{\mathbb{R}^+}^{}\bar\rho\varphi_x\zeta {\rm d}x\notag\\
			&=\frac{\mu}{2a}\int_{\mathbb{R}^+}^{}\bar\rho_t\zeta^2{\rm d}x-\frac{\mu D}{a}\int_{\mathbb{R}^+}^{}\bar\rho_x\zeta_x\zeta {\rm d}x+\frac{\mu}{a}\int_{\mathbb{R}^+}^{}\bar\rho\zeta g {\rm d}x.
		\end{align}
		Adding up \eqref{31} and \eqref{32} shows
		\begin{align}\label{33}
			&\;\;\;\;\frac{{\rm d}}{{\rm d}t}\left\{\frac{\alpha}{2}\int_{\mathbb{R}^+}\varphi^2{\rm d}x+\int_{\mathbb{R}^+}^{}\varphi_t\varphi{\rm d}x+\frac{\mu}{2a}\int_{\mathbb{R}^+}^{}\bar\rho\zeta^2{\rm d}x\right\}+\frac{\mu D}{a}\int_{\mathbb{R}^+}^{}\bar\rho\zeta_x^2{\rm d}x	\notag\\
			&\;\;\;\;+\int_{\mathbb{R}^+}^{}p^{\prime}(\bar\rho)\varphi_x^2{\rm d}x-2\mu\int_{\mathbb{R}^+}^{}\bar\rho\zeta\varphi_x{\rm d}x+\frac{b\mu}{a}\int_{\mathbb{R}^+}^{}\bar\rho\zeta^2{\rm d}x\notag\\
			&=\int_{\mathbb{R}^+}\varphi_t^2{\rm d}x-\mu\int_{\mathbb{R}^+}^{}\varphi_x\zeta_x\varphi{\rm d}x-\mu\int_{\mathbb{R}^+}^{}\varphi_x\hat\phi_x\varphi{\rm d}x-\mu\int_{\mathbb{R}^+}^{}\hat\rho\zeta_x\varphi{\rm d}x\notag\\
			&\;\;\;\;-\mu\int_{\mathbb{R}^+}\bar\rho\hat\phi_x\varphi{\rm d}x-\mu\int_{\mathbb{R}^+}\hat\rho\bar\phi_x\varphi{\rm d}x-\mu\int_{\mathbb{R}^+}\hat\rho\hat\phi_x\varphi{\rm d}x+\frac{\mu}{2a}\int_{\mathbb{R}^+}^{}\bar\rho_t\zeta^2{\rm d}x\notag\\
			&\;\;\;\;-\frac{\mu D}{a}\int_{\mathbb{R}^+}^{}\bar\rho_x\zeta_x\zeta {\rm d}x+\frac{\mu}{a}\int_{\mathbb{R}^+}^{}\bar\rho\zeta g{\rm d}x+\int_{\mathbb{R}^+}^{}f\varphi_x{\rm d}x-\int_{\mathbb{R}^+}^{}h_x\varphi{\rm d}x\notag\\
			&\;\;\;\;-\frac{\mu}{b}\int_{\mathbb{R}^+}[\bar\rho_x(a\varphi_x-b\zeta)]\varphi{\rm d}x:=\int_{\mathbb{R}^+}\varphi_t^2{\rm d}x+\sum_{i=1}^{12}I_i.
		\end{align}
		Next, we estimate the right side of \eqref{33} one by one.
		From \eqref{13.18} and the Cauchy inequality, we have
		\begin{equation}\label{I1}
			I_1
			\lesssim\|\varphi\|_{L^{\infty}}\int_{\mathbb{R}^+}|\varphi_x\zeta_x|{\rm d}x
			\lesssim\varepsilon\left(\|\varphi_x\|^2+\|\zeta_x\|^2\right).
		\end{equation}
		By applying Lemma \ref{l3.2} and  \eqref{13.18}, we obatin
		\begin{equation}\label{I26}
			\sum_{i=2}^{6}I_i\lesssim \|\hat\phi_x\|_{L^{\infty}}\|\varphi\|_1+\|\hat\rho\|_{L^{\infty}}\|\zeta_x\|\|\varphi\|+\|\hat\phi_x\|\|\varphi\|+\|\hat\rho\|_{L^{\infty}}\|\bar\phi_x\|\|\varphi\|+\|\hat\rho\|_{L^{\infty}}\|\hat\phi_x\|\|\varphi\|\lesssim \delta{\rm e}^{-ct}.
		\end{equation}
		Recalling the definition of $f$ and $g$ in \eqref{22}, and using Lemmas \ref{l3.1}--\ref{l3.2}, Taylor's formula and the Cauchy inequality give
		\begin{align}\label{I789}
			I_{7}+I_{8}+I_9&\lesssim\|\bar\rho_t\|_{L^{\infty}}\|\zeta\|^2+\|\bar\rho_x\|_{L^{\infty}}\|\zeta_x\|\|\zeta\|+\|\zeta\|\|g\|\notag\\
			 &\lesssim \delta\left(\|\zeta\|^2+\|\zeta_x\|^2\right)+\frac{1}{\delta}\|g\|^2\notag\\&\lesssim\delta\left(\|\zeta\|^2+\|\zeta_x\|^2\right)+\delta(1+t)^{-\frac{3}{2}},
		\end{align}
		\begin{equation}\label{I10}
			I_{10}\lesssim \int_{\mathbb{R^+}}|\varphi_x|(|\bar\rho_t|+|\varphi_x|^2+|\hat \rho|){\rm d}x\lesssim(\varepsilon+\delta)\|\varphi_x\|^2+\delta(1+t)^{-\frac{3}{2}}.
		\end{equation}
		For  $I_{11}$, notice that
		\begin{align}\label{hx}
			h_x&=-\left(\frac{\left(\varphi_t+\frac{1}{\alpha}q(\bar\rho)_x-\hat m\right)^2}{\varphi_x+\bar\rho+\hat\rho}\right)_x\nonumber\\
			&=-\frac{2\left(\varphi_t+\frac{1}{\alpha}q(\bar\rho)_x-\hat m\right)}{\varphi_x+\bar\rho+\hat\rho}\left(\varphi_{xt}+\frac{1}{\alpha}q(\bar\rho)_{xx}-\hat m_x\right)+\frac{\left(\varphi_t+\frac{1}{\alpha}q(\bar\rho)_x-\hat m\right)^2}{\left(\varphi_x+\bar\rho+\hat\rho\right)^2}\left(\varphi_{xx}+\bar\rho_{x}+\hat\rho_x\right),
		\end{align}
		then it follows from Lemmas \ref{l3.1}--\ref{l3.2} and \eqref{13.15} that
		\begin{align}\label{I11}
			I_{11}\lesssim&
			\int_{\mathbb{R}^+}\left(|\varphi_t|+|\bar\rho_x|+|\hat m|\right)\left(|\varphi_{xt}|+|\bar\rho_{xx}|+|\bar\rho_x|^2+|\hat m_x|^2\right)|\varphi|{\rm d}x\notag\\
			&+\int_{\mathbb{R}^+}\left(|\varphi_t|^2+|\bar\rho_x|^2+|\hat m|^2\right)\left(|\varphi_{xx}|+|\bar\rho_{x}|+|\hat\rho_x|\right)|\varphi|{\rm d}x\notag\\
			\lesssim&(\varepsilon+\delta)\left(\|\varphi_x\|_1^2+\|\varphi_t\|_1^2\right)+{\rm e}^{-ct}\left(\|\varphi\|^2+\|\varphi_{xx}\|^2+\|\varphi_{xt}\|^2\right) +\delta(1+t)^{-\frac{5}{4}}.
		\end{align}
		For $I_{12}$, we can similarly prove that
		\begin{align}\label{I12}
			I_{12}&=-\frac{\mu}{b}\int_{\mathbb{R}^+}\bar\rho_x(a\varphi_x-b\zeta)\varphi{\rm d}x=-\frac{\mu}{b}\int_{\mathbb{R}^+}\bar\rho_x(\zeta_t-D\zeta_{xx}-g)\varphi{\rm d}x\notag\\
			&=-\frac{\mu}{b}	\frac{{\rm d}}{{\rm d}t}\int_{\mathbb{R}^+}\bar\rho_x\zeta \varphi{\rm d}x+\frac{\mu}{b}\int_{\mathbb{R}^+}\bar\rho_{xt}\zeta \varphi{\rm d}x+\frac{\mu}{b}\int_{\mathbb{R}^+}\bar\rho_x\zeta \varphi_t{\rm d}x\notag\\
			&\;\;\;\;-\frac{\mu D}{b}\int_{\mathbb{R}^+}\bar\rho_{xx}\zeta_x\varphi{\rm d}x-\frac{\mu D}{b}\int_{\mathbb{R}^+}\bar\rho_x\zeta_x\varphi_x{\rm d}x+\frac{\mu}{b}\int_{\mathbb{R}^+}\bar\rho_x(-\bar\phi_t+D\bar\phi_{xx}+D\hat\phi_{xx})\varphi{\rm d}x\notag\\
			&\leq -\frac{\mu}{b}	\frac{{\rm d}}{{\rm d}t}\int_{\mathbb{R}^+}\bar\rho_x\zeta \varphi{\rm d}x+C\|\varphi\|_{L^{\infty}}\|\bar\rho_{xt}\|\|\zeta\|+C\|\bar\rho_x\|_{L^{\infty}}\|\zeta\|\|\varphi_t\|\notag\\	
			&\;\;\;\;+C\|\varphi\|_{L^{\infty}}\|\bar\rho_{xx}\|\|\zeta_x\|+C\|\bar\rho_x\|_{L^{\infty}}\|\zeta_x\|\|\varphi_x\|\notag\\
			&\;\;\;\;+C\|\bar\rho_x\|_{L^{\infty}}\|\varphi\|\left(\|\bar\phi_t\|+\|\bar\phi_{xx}\|+\|\hat\phi_{xx}\|\right)\notag\\
			&\leq  -\frac{\mu}{b} 	\frac{{\rm d}}{{\rm d}t} \int_{\mathbb{R}^+}\bar\rho_x\zeta \varphi{\rm d}x+C(\varepsilon+\delta)\left(\|\varphi_x\|^2+\|\varphi_t\|^2+\|\zeta\|_1^2\right) +C\delta(1+t)^{-\frac{5}{4}}.
		\end{align}
		Substituting  \eqref{I1}--\eqref{I12}into \eqref{33}, we obtain
		\begin{align}\label{34}
			&\;\;\;\;	\frac{{\rm d}}{{\rm d}t}\left\{\frac{\alpha}{2}\int_{\mathbb{R}^+}\varphi^2{\rm d}x+\int_{\mathbb{R}^+}\varphi_t\varphi{\rm d}x+\frac{\mu}{2a}\int_{\mathbb{R}^+}\bar\rho\zeta^2{\rm d}x\right\}+\frac{\mu D}{a}\int_{\mathbb{R}^+}\bar\rho\zeta_x^2{\rm d}x\notag\\
			&\;\;\;\;+\int_{\mathbb{R}^+}p^{\prime}(\bar\rho)\varphi_x^2-2\mu\int_{\mathbb{R}^+}\bar\rho\zeta\varphi_x{\rm d}x+\frac{b\mu}{a}\int_{\mathbb{R}^+}\bar\rho\zeta^2{\rm d}x\notag\\
			&\leq -\frac{\mu}{b}	\frac{{\rm d}}{{\rm d}t}\int_{\mathbb{R}^+}\bar\rho_x\zeta \varphi{\rm d}x+\|\varphi_t\|_2^2+C(\varepsilon+\delta)\left(\|\varphi_x\|_1^2+\|\varphi_t\|_1^2+\|\zeta\|_1^2\right)+C\delta(1+t)^{-\frac{5}{4}}\notag\\
			&\;\;\;\;+C{\rm e}^{-ct}\left(\|\varphi\|_2^2+\|\varphi_{xt}\|^2\right).
		\end{align}
		Next, by calculating $\int_{\mathbb{R}^+}\left(\eqref{29}_1\times\varphi_t\right){\rm d}x$, we have
		\begin{align}\label{36}
			&\;\;\;\;	\frac{{\rm d}}{{\rm d}t}\left\{\frac{1}{2}\int_{\mathbb{R}^+}\varphi_t^2{\rm d}x+\frac{1}{2}\int_{\mathbb{R}^+}p^{\prime}(\bar\rho)\varphi_x^2{\rm d}x-\mu\int_{\mathbb{R}^+}\bar\rho\zeta\varphi_x{\rm d}x\right\}+\alpha\int_{\mathbb{R}^+}\varphi_t^2{\rm d}x+\mu\int_{\mathbb{R}^+}\bar\rho\zeta_t\varphi_x{\rm d}x\notag\\
			&=\frac{1}{2}\int_{\mathbb{R}^+}p^{\prime\prime}(\bar\rho)\bar\rho_t\varphi_x^2{\rm d}x-\mu\int_{\mathbb{R}^+}\bar\rho_t\zeta\varphi_x{\rm d}x-\mu\int_{\mathbb{R}^+}\varphi_x\zeta_x\varphi_t{\rm d}x-\frac{\mu}{b}\int_{\mathbb{R}^+}\bar\rho_x(a\varphi_x-b\zeta)\varphi_t{\rm d}x\notag\\
			&\;\;\;\;-\mu\int_{\mathbb{R}^+}\varphi_x\hat\phi_x\varphi_t{\rm d}x-\mu\int_{\mathbb{R}^+}\hat\rho\zeta_x\varphi_t{\rm d}x-\mu\int_{\mathbb{R}^+}\bar\rho\hat\phi_x\varphi_t{\rm d}x\notag\\
			&\;\;\;\;-\mu\int_{\mathbb{R}^+}\hat\rho\bar\phi_x\varphi_t{\rm d}x-\mu\int_{\mathbb{R}^+}\hat\rho\hat\phi_x\varphi_t{\rm d}x-\int_{\mathbb{R}^+}f_x\varphi_t{\rm d}x-\int_{\mathbb{R}^+}h_x\varphi_t{\rm d}x,
		\end{align}
		where we used 
		\begin{align}\notag
		\int_{\mathbb{R}^+}(\bar\rho\zeta)_x\varphi_t{\rm d}x=-\int_{\mathbb{R}^+}\bar\rho\zeta\varphi_{xt}{\rm d}x=-\frac{{\rm d}}{{\rm d}t}\int_{\mathbb{R}^+}\bar\rho\zeta\varphi_{x}{\rm d}x+\int_{\mathbb{R}^+}\bar\rho\zeta_t\varphi_x{\rm d}x+\int_{\mathbb{R}^+}\bar\rho_t\zeta\varphi_x{\rm d}x.	
		\end{align}
		Then we calculate
		$\int_{\mathbb{R}^+}\left(\eqref{29}_2\times\left(\frac{\mu}{a}\bar\rho\zeta_t\right)\right){\rm d}x$ to give
		\begin{align}\label{37}
			&\;\;\;\;\;	\frac{{\rm d}}{{\rm d}t}\left\{\frac{\mu b}{2a}\int_{\mathbb{R}^+}^{}\bar\rho\zeta^2{\rm d}x+\frac{\mu D}{2a}\int_{\mathbb{R}^+}^{}\bar\rho\zeta_x^2{\rm d}x\right\}
			+\frac{\mu}{a}\int_{\mathbb{R}^+}^{}\bar\rho\zeta_t^2{\rm d}x-\mu\int_{\mathbb{R}^+}^{}\bar\rho\varphi_x\zeta_t {\rm d}x\notag\\
			&=\frac{\mu b}{2a}\int_{\mathbb{R}^+}^{}\bar\rho_t\zeta^2{\rm d}x+\frac{\mu D}{2a}\int_{\mathbb{R}^+}^{}\bar\rho_t\zeta_x^2{\rm d}x
			-\frac{\mu D}{a}\int_{\mathbb{R}^+}^{}\bar\rho_x\zeta_x\zeta_t {\rm d}x+\frac{\mu}{a}\int_{\mathbb{R}^+}^{}\bar\rho\zeta_tg {\rm d}x.
		\end{align}
		Adding \eqref{36} and \eqref{37}, and noticing that $\mu\int_{\mathbb{R}^+}^{}\bar\rho\varphi_x\zeta_t {\rm d}x$ is deleted, we have
		\begin{align}\label{38}
			&\;\;\;\;	\frac{{\rm d}}{{\rm d}t}\left\{\frac{1}{2}\int_{\mathbb{R}^+}\varphi_t^2{\rm d}x	+\frac{\mu D}{2a}\int_{\mathbb{R}^+}^{}\bar\rho\zeta_x^2{\rm d}x\right\}+\alpha\int_{\mathbb{R}^+}\varphi_t^2{\rm d}x+\frac{\mu}{a}\int_{\mathbb{R}^+}^{}\bar\rho\zeta_t^2{\rm d}x\notag\\
			&\;\;\;\;+\frac{1}{2}	\frac{{\rm d}}{{\rm d}t}\left\{\int_{\mathbb{R}^+}p^{\prime}(\bar\rho)\varphi_x^2{\rm d}x-2\mu\int_{\mathbb{R}^+}\bar\rho\zeta\varphi_x{\rm d}x+\frac{\mu b}{a}\int_{\mathbb{R}^+}^{}\bar\rho\zeta^2{\rm d}x\right\}\notag\\
			&=\frac{1}{2}\int_{\mathbb{R}^+}p^{\prime\prime}(\bar\rho)\bar\rho_t\varphi_x^2{\rm d}x-\mu\int_{\mathbb{R}^+}\bar\rho_t\zeta\varphi_x{\rm d}x+\frac{\mu b}{2a}\int_{\mathbb{R}^+}^{}\bar\rho_t\zeta^2{\rm d}x+\frac{\mu D}{2a}\int_{\mathbb{R}^+}^{}\bar\rho_t\zeta_x^2{\rm d}x
			-\frac{\mu D}{a}\int_{\mathbb{R}^+}^{}\bar\rho_x\zeta_x\zeta_t {\rm d}x\notag\\
			&\;\;\;\;-\frac{\mu}{b}\int_{\mathbb{R}^+}[\bar\rho_x(a\varphi_x-b\zeta)]\varphi_t{\rm d}x-\mu\int_{\mathbb{R}^+}\varphi_x\zeta_x\varphi_t{\rm d}x-\mu\int_{\mathbb{R}^+}\varphi_x\hat\phi_x\varphi_t{\rm d}x\notag\\
			&\;\;\;\;	-\mu\int_{\mathbb{R}^+}\hat\rho\zeta_x\varphi_t{\rm d}x-\mu\int_{\mathbb{R}^+}\bar\rho\hat\phi_x\varphi_t{\rm d}x-\mu\int_{\mathbb{R}^+}\hat\rho\bar\phi_x\varphi_t{\rm d}x-\mu\int_{\mathbb{R}^+}\hat\rho\hat\phi_x\varphi_t{\rm d}x\notag\\
			&\;\;\;\;+\frac{\mu}{a}\int_{\mathbb{R}^+}^{}\bar\rho\zeta_tg {\rm d}x-\int_{\mathbb{R}^+}f_x\varphi_t{\rm d}x-\int_{\mathbb{R}^+}h_x\varphi_t{\rm d}x:=\sum_{i=13}^{27}I_i.
		\end{align}
		Using Lemmas \ref{l3.1}--\ref{l3.2}, \eqref{13.15}, \eqref{13.18} and Young inequality, we can deduce
		\begin{align}\label{J18}
			\sum_{i=13}^{18}I_i\lesssim&
			\left(\|\bar\rho_t\|_{L^{\infty}}+\frac{1}{\delta}\|\bar\rho_x\|_{L^{\infty}}^2\right)\left(\|\varphi_x\|^2+\|\zeta\|^2+\|\zeta_x\|^2\right)
			+\delta\left(\|\varphi_t\|^2+\|\zeta_t\|^2\right)
			\notag\\
			\lesssim&\delta(1+t)^{-1}\left(\|\varphi_x\|^2+\|\zeta\|^2+\|\zeta_x\|^2\right)
			+\delta\left(\|\varphi_t\|^2+\|\zeta_t\|^2\right),
		\end{align}
		\begin{align}\label{K9}
			I_{19}\leq \frac{\alpha}{8}\|\varphi_t\|^2+C\|\zeta_x\|_{L^{\infty}}^2\|\varphi_x\|^2\leq \frac{\alpha}{8}\|\varphi_t\|^2+C\varepsilon^2(1+t)^{-2}\|\varphi_x\|^2.
		\end{align}
		By Lemmas \ref{l3.1}--\ref{l3.2} and \eqref{13.15}, we can also prove
		\begin{align}
			\sum_{i=20}^{24}I_i\leq C\delta{\rm e}^{-ct}.
		\end{align}
		Recalling
		the definition of $g$ in \eqref{22}, and applying Lemmas \ref{l3.1}--\ref{l3.2} lead to
		\begin{align}\label{K15}
			I_{25}
			&=\frac{\mu}{a}\frac{{\rm d}}{{\rm d}t}\int_{\mathbb{R}^+}\bar\rho\zeta g{\rm d}x-\frac{\mu}{a}\int_{\mathbb{R}^+}\bar\rho_t\zeta g{\rm d}x-\frac{\mu}{a}\int_{\mathbb{R}^+}\bar\rho\zeta g_t{\rm d}x\notag\\
			&\leq \frac{\mu}{a}\frac{{\rm d}}{{\rm d}t}\int_{\mathbb{R}^+}\bar\rho\zeta g{\rm d}x+C\delta(1+t)^{-1}\|\zeta\|^2
			+C\frac{1}{\delta}(1+t)\int_{\mathbb{R}^+}(\bar\rho^2g_t^2+\bar\rho_t^2g^2){\rm d}x\notag\\
			&\leq \frac{\mu}{a}\frac{{\rm d}}{{\rm d}t}\int_{\mathbb{R}^+}\bar\rho\zeta g{\rm d}x+C\delta(1+t)^{-1}\|\zeta\|^2+C\delta(1+t)^{-\frac{5}{2}}.
		\end{align}
	For $I_{26}$--$I_{27}$, notice that
		\begin{align}\label{fx}
			f_x&=\frac{1}{\alpha}q(\bar\rho)_{xt}-\left(p(\varphi_x+\bar\rho+\hat\rho)-p(\bar\rho)-p^{\prime}(\bar\rho)\varphi_x\right)_x\notag\\
			&=\frac{1}{\alpha}q(\bar\rho)_{xt}-(p^{\prime}(\varphi_x+\bar\rho+\hat\rho)-p^{\prime}(\bar\rho))\varphi_{xx}-p^{\prime}(\varphi_x+\bar\rho+\hat\rho)\hat\rho_x\notag\\
			&\;\;\;\;-(p^{\prime}(\varphi_x+\bar\rho+\hat\rho)-p^{\prime}(\bar\rho)-p^{\prime\prime}(\bar\rho)\varphi_x)\bar\rho_x.
		\end{align}
	Through direct calculation, we can get
		\begin{align}\label{J16}
			I_{26}\leq& -\frac{1}{2}	\frac{{\rm d}}{{\rm d}t}\int_{\mathbb{R}^+}\left(p^{\prime}(\varphi_x+\bar\rho+\hat\rho)-p^{\prime}(\bar\rho)\right)\varphi_x^2{\rm d}x+C(\|\bar\rho_x\|_{L^\infty}\|\bar\rho_t\|+\|\bar\rho_{xt}\|+\|\bar\rho_x\|_{L^\infty}\|\varphi_x\|+\|\hat\rho\|_1)\|\varphi_t\|\notag\\
			&+C(\|\varphi_{xt}\|_{L^\infty}+\|\bar\rho_t\|_{L^\infty}+\|\hat\rho_t\|_{L^\infty})\|\varphi_x\|^2\notag\\ 
			\leq&-\frac{1}{2}	\frac{{\rm d}}{{\rm d}t}\int_{\mathbb{R}^+}\left(p^{\prime}(\varphi_x+\bar\rho+\hat\rho)-p^{\prime}(\bar\rho)\right)\varphi_x^2{\rm d}x+C(\varepsilon+\delta)((1+t)^{-1}\|\varphi_x\|^2+\|\varphi_t\|^2)\notag\\ 
			&+C\delta(1+t)^{-\frac{5}{2}},
		\end{align}
		\begin{align}\label{J17}
			I_{27}
			=&\int_{\mathbb{R}^+}\left[\left(\frac{2(\varphi_t+\frac{1}{\alpha}q(\bar\rho)_x-\hat m)}{(\varphi_x+\bar\rho+\hat\rho)}\right)\left(\varphi_{xt}+\frac{1}{\alpha}q(\bar\rho)_{xx}-\hat m_x\right)\right]\varphi_t{\rm d}x\notag\\
			&-\int_{\mathbb{R}^+}\left[\left(\frac{(\varphi_t+\frac{1}{\alpha}q(\bar\rho)_x-\hat m)^2}{(\varphi_x+\bar\rho+\hat\rho)^2}\right)(\varphi_{xx}+\bar\rho_x+\hat\rho_x	)\right]\varphi_t{\rm d}x\notag\\
			\leq& C(\|\varphi_{xt}+\frac{1}{\alpha}q(\bar\rho)_{xx}-\hat m_x\|_{L^\infty}+\|\varphi_t\|_{L^\infty})(\|\varphi_t\|+\|\bar\rho_x\|)\|\varphi_t\|+C\|\bar\rho_x\|^2_{L^\infty}\|\varphi_t\|(\|\varphi_{xx}\|+\|\bar\rho_{x}\|+\|\hat \rho_x\|)\notag\\
			&+C\|\hat m\|_{L^\infty}\|\varphi_t\|(\|\varphi_{xt}\|+\|q(\bar\rho)_{xx}\|+\|\hat m_x\|)+C\|\hat m\|^2_{L^\infty}\|\varphi_t\|(\|\varphi_{xx}\|+\|\bar\rho_{x}\|+\|\hat \rho_x\|)\notag\\
			\leq& \frac{\alpha}{8}\|\varphi_t\|^2+C{\rm e}^{-ct}(\|\varphi_{x}\|_1^2+\|\varphi_{xt}\|^2)+C\delta(1+t)^{-\frac{5}{2}}.
		\end{align}
	Putting \eqref{J18}--\eqref{J17} into \eqref{38}, and using the smallness of $\varepsilon$ and $\delta$ show	
\begin{align}\label{13.43}
			&\frac{{\rm d}}{{\rm d}t}\left\{\frac{1}{2}\int_{\mathbb{R}^+}\varphi_t^2{\rm d}x	+\frac{\mu D}{2a}\int_{\mathbb{R}^+}^{}\bar\rho\zeta_x^2{\rm d}x\right\}+\frac{\alpha}{2}\int_{\mathbb{R}^+}\varphi_t^2{\rm d}x+\frac{\mu}{2a}\int_{\mathbb{R}^+}^{}\bar\rho\zeta_t^2{\rm d}x\notag\\
			&+\frac{1}{2}	\frac{{\rm d}}{{\rm d}t}\left\{\int_{\mathbb{R}^+}p^{\prime}(\bar\rho)\varphi_x^2{\rm d}x-2\mu\int_{\mathbb{R}^+}\bar\rho\zeta\varphi_x{\rm d}x+\frac{\mu b}{a}\int_{\mathbb{R}^+}^{}\bar\rho\zeta^2{\rm d}x\right\}\notag\\
			\leq &\frac{\mu}{a}\frac{{\rm d}}{{\rm d}t}\int_{\mathbb{R}^+}\bar\rho\zeta g{\rm d}x-\frac{1}{2}	\frac{{\rm d}}{{\rm d}t}\int_{\mathbb{R}^+}\left(p^{\prime}(\varphi_x+\bar\rho+\hat\rho)-p^{\prime}(\bar\rho)\right)\varphi_x^2{\rm d}x+(\varepsilon+\delta)(1+t)^{-1}\left(\|\varphi_x\|^2+\|\zeta\|_1^2\right)\notag\\ 
			&+C{\rm e}^{-ct}(\|\varphi_{x}\|_1^2+\|\varphi_{xt}\|^2)+C\delta(1+t)^{-\frac{5}{2}}.
		\end{align}
Adding up \eqref{34} and \eqref{13.43}, then integrating the resulting inequality with respect to $t$ and using \eqref{23} give
		\begin{align}\label{13.44}
			& \|[\varphi,\varphi_x,\varphi_t,\zeta,\zeta_x](t)\|^2+\int_{0}^{t}\|[\varphi_x,\varphi_\tau,\zeta,\zeta_x,\zeta_{\tau}](\tau)\|^2{\rm d}\tau \notag\\
			&	\leq C\left(\|\varphi_0\|_1^2+\|\psi_0\|^2+\|\zeta_0\|_1^2+\delta\right)+C\int_{0}^{t}{\rm e}^{-c\tau}\left(\|\varphi(\tau)\|_2^2+\|\varphi_{\tau}(\tau)\|_1^2\right){\rm d}\tau.
		\end{align}
Then integrating $(1 + t) \times \eqref{13.43} $ over $[0, t]$ and using \eqref{13.44} yield
		\begin{align}\label{13.45}
			&(1+t)\|[\varphi_x,\varphi_t,\zeta,\zeta_x](t)\|^2+\int_{0}^{t}(1+\tau)\|[\varphi_{\tau},\zeta_{\tau}](\tau)\|^2{\rm d}\tau\notag\\ 
			&\leq C\left(\|\varphi_0\|_1^2+\|\psi_0\|^2+\|\zeta_0\|_1^2+\delta\right)+C\int_{0}^{t}{\rm e}^{-c\tau}\left(\|\varphi(\tau)\|_2^2+\|\varphi_{\tau}(\tau)\|_1^2\right){\rm d}\tau,
		\end{align}
		where we used 
\begin{align}\notag
	(1+t)\int_{\mathbb{R}^+}|\bar\rho\zeta g|{\rm d}x&\lesssim(1+t)(\delta\|\zeta\|^2+\frac{1}{\delta}\|g\|^2)\lesssim\delta(1+t)\|\zeta\|^2+\delta(1+t)^{-\frac{1}{2}}\notag\\ 
	&\lesssim\delta(1+t)\|\zeta\|^2+\delta.\notag
\end{align}
\eqref{43} follows from \eqref{13.44}--\eqref{13.45}. The proof of this lemma is completed.
		\end{proof}
	\begin{lemma}\label{l3.4}
		Under the hypothesis of Proposition \ref{p3.1}, it holds that
		\begin{align}\label{13.46}
			&\;\;\;\;(1+t)^2\|[\varphi_{xt},\varphi_{xx},\zeta_x,\zeta_{xx}](t)\|^2+\int_{0}^{t}\left\{(1+\tau)^2\|[\varphi_{x\tau},\zeta_{x\tau}](\tau)\|^2+(1+\tau)\|[\varphi_{xx},\zeta_x,\zeta_{xx}](\tau)\|^2\right\}{\rm d}\tau\notag\\
			&\leq C(\|\varphi_0\|_3^2+\|\psi_0\|_2^2+\|\zeta_0\|_3^2+\delta)+C\int_{0}^{t}{\rm e}^{-c\tau}\left(\|\varphi(\tau)\|_3^2+\|\varphi_{\tau}(\tau)\|_1^2\right){\rm d}\tau.
		\end{align}
	\end{lemma}
	\begin{proof}
By calculating $\int_{\mathbb{R}^+}\left(\partial_x\eqref{29}_1\times\varphi_x\right){\rm d}x+\int_{\mathbb{R}^+}\left(\partial_x\eqref{29}_2\times\left(\frac{\mu}{a}\bar\rho\zeta_x\right)\right){\rm d}x$, we obtain
\begin{align}\label{13.47}
			&\frac{{\rm d}}{{\rm d}t}\left\{\frac{\alpha}{2}\int_{\mathbb{R}^+}\varphi_x^2{\rm d}x+\int_{\mathbb{R}^+}^{}\varphi_{xt}\varphi_x{\rm d}x+\frac{\mu}{2a}\int_{\mathbb{R}^+}^{}\bar\rho\zeta_x^2{\rm d}x\right\}+\frac{\mu D}{a}\int_{\mathbb{R}^+}^{}\bar\rho\zeta_{xx}^2{\rm d}x	\notag\\
			&+\int_{\mathbb{R}^+}^{}p^{\prime}(\bar\rho)\varphi_{xx}^2{\rm d}x-2\mu\int_{\mathbb{R}^+}^{}\bar\rho\zeta_x\varphi_{xx}{\rm d}x+\frac{b\mu}{a}\int_{\mathbb{R}^+}^{}\bar\rho\zeta_x^2{\rm d}x\notag\\
			=&\int_{\mathbb{R}^+}\varphi_{xt}^2{\rm d}x-\mu\int_{\mathbb{R}^+}^{}\partial_x(\varphi_{x}\zeta_x)\varphi_{x} {\rm d}x-\frac{\mu}{b}\int_{\mathbb{R}^+}^{}\partial_x[\bar\rho_x(a\varphi_{x}-b\zeta)]\varphi_{x}{\rm d}x-\mu\int_{\mathbb{R}^+}^{}\partial_x(\varphi_{x}\hat\phi_x)\varphi_{x}{\rm d}x\notag\\
			&\;\;\;\;-\mu\int_{\mathbb{R}^+}^{}\partial_x(\hat\rho\zeta_x)\varphi_{x}{\rm d}x	-\mu\int_{\mathbb{R}^+}^{}\partial_x(\bar\rho\hat\phi_x)\varphi_{x}{\rm d}x-\mu\int_{\mathbb{R}^+}^{}\partial_x(\hat\rho\bar\phi_x)\varphi_{x}{\rm d}x	-\mu\int_{\mathbb{R}^+}^{}\partial_x(\hat\rho\hat\phi_x)\varphi_{x}{\rm d}x\notag\\
			&
			-\int_{\mathbb{R}^+}^{}f_{xx}\varphi_{x}{\rm d}x-\int_{\mathbb{R}^+}^{}p^{\prime}(\bar\rho)_x\varphi_{x}\varphi_{xx}{\rm d}x
			-\int_{\mathbb{R}^+}^{}h_{xx}\varphi_{x}{\rm d}x
			+\mu\int_{\mathbb{R}^+}^{}\bar\rho_x\zeta\varphi_{xx}{\rm d}x\notag\\ 
			&+\frac{\mu}{2a}\int_{\mathbb{R}^+}^{}\bar\rho_t\zeta_x^2{\rm d}x-\frac{\mu D}{a}\int_{\mathbb{R}^+}^{}\bar\rho_x\zeta_{xx}\zeta_{x} {\rm d}x+\frac{\mu}{a}\int_{\mathbb{R}^+}^{}\bar\rho\zeta_{x}g_x {\rm d}x:=\int_{\mathbb{R}^+}\varphi_{xt}^2{\rm d}x+\sum_{i=1}^{14}J_i.
		\end{align}
By applying Lemmas \ref{l3.1}--\ref{l3.2}, \eqref{13.15}, \eqref{13.18} and Young inequality, we can deduce
\begin{align}\label{13.48}
			J_{1}\lesssim&\int_{\mathbb{R}^+}|\varphi_x\zeta_x\varphi_{xx}|{\rm d}x+\int_{\mathbb{R}^+}|\varphi_x^2\zeta_{xx}|{\rm d}x\notag\\
			\lesssim&\|\varphi_x\|_{L^{\infty}}\int_{\mathbb{R}^+}|\zeta_x\varphi_{xx}|{\rm d}x+\|\varphi_x\|_{L^{\infty}}\int_{\mathbb{R}^+}|\zeta_{xx}\varphi_{x}|{\rm d}x\notag\\
			\lesssim&\varepsilon(\|\varphi_{xx}\|^2+\|\zeta_{xx}\|^2)+\varepsilon(1+t)^{-1}(\|\varphi_x\|^2+\|\zeta_x\|^2),
		\end{align}
\begin{align}\label{13.49}
			J_{2}=&-\frac{\mu}{b}\int_{\mathbb{R}^+}\bar\rho_{xx}\varphi_x(a\varphi_x-b\zeta){\rm d}x-\frac{\mu}{b}\int_{\mathbb{R}^+}\bar\rho_{x}\varphi_x(a\varphi_{xx}-b\zeta_{x}){\rm d}x\notag\\ 
			\lesssim&\|\bar\rho_{xx}\|_{L^{\infty}}\|\varphi_x\|(\|\varphi_x\|+\|\zeta\|)+\|\bar\rho_x\|_{L^{\infty}}\|\varphi_{xx}\|(\|\varphi_{x}\|+\|\zeta\|)\notag\\	
			\lesssim&\delta\|\varphi_{xx}\|^2+\delta(1+t)^{-1}\left(\|\varphi_x\|^2+\|\zeta\|^2\right),
		\end{align}
		\begin{align}\label{13.50}
			\sum_{i=3}^{7}J_i\lesssim \delta{\rm e}^{-ct},
		\end{align}
\begin{align}\label{13.51}
J_{8}+J_{9}\lesssim&\int_{\mathbb{R}^+}|(p^\prime(\varphi_{x}+\bar{\rho}+\hat{\rho})-p^\prime(\bar{\rho}))(\varphi_{xx}+\bar{\rho}_{x})\varphi_{xx}|{\rm d}x+\int_{\mathbb{R}^+}(|\hat{\rho}_{x}|+|\bar{\rho}_{x}\bar{\rho}_{t}|+|\bar{\rho}_{xt}|)|\varphi_{xx}|{\rm d}x\nonumber\\
   &+\int_{\mathbb{R}^+}|\bar{\rho}_{x}\varphi_{x}\varphi_{xx}|{\rm d}x\nonumber\\
   \lesssim& (\varepsilon+\delta)\|\varphi_{xx}\|^{2}+\delta(1+t)^{-1}\|\varphi_{x}\|^{2}+\delta (1+t)^{-\frac{5}{2}},	
\end{align}
\begin{align}\label{13.52}
J_{10}\lesssim&
			\int_{\mathbb{R}^+}\left(|\varphi_t|+|\bar\rho_x|+|\hat m|\right)\left(|\varphi_{xt}|+|\bar\rho_{xx}|+|\bar\rho_x|^2+|\hat m_x|\right)|\varphi_{xx}|{\rm d}x\notag\\
			&+\int_{\mathbb{R}^+}\left(|\varphi_t|^2+|\bar\rho_x|^2+|\hat m|^2\right)\left(|\varphi_{xx}|+|\bar\rho_{x}|+|\hat\rho_x|\right)|\varphi_{xx}|{\rm d}x\notag\\
			\lesssim&(\varepsilon+\delta)\|\varphi_{xx}\|^2+{\rm e}^{-ct}\left(\|\varphi_{xx}\|^2+\|\varphi_{xt}\|^2\right) +\delta(1+t)^{-\frac{5}{2}},	
\end{align}
and
\begin{align}\label{13.53}
			\sum_{i=11}^{14}J_i\lesssim&\|\bar\rho_x\|_{L^{\infty}}(\|\zeta\|\|\varphi_{xx}\|+\|\zeta\|\|g_{x}\|+\|\zeta_x\|\|\zeta_{xx}\|)+\|\bar\rho_t\|_{L^{\infty}}\|\zeta_x\|^2+\|\zeta\|\|g_{xx}\|\notag\\
			\lesssim&\delta(\|\varphi_{xx}\|^2+\|\zeta_{xx}\|^2)+\delta(1+t)^{-1}\left(\|\varphi_x\|^2+\|\zeta\|^2+\|\zeta_x\|^2\right)+\delta (1+t)^{-\frac{5}{2}}.
		\end{align}
	Substituting \eqref{13.48}--\eqref{13.53} into \eqref{13.47}, we deduce
\begin{align}\label{13.54}
			&\frac{{\rm d}}{{\rm d}t}\left\{\frac{\alpha}{2}\int_{\mathbb{R}^+}\varphi_x^2{\rm d}x+\int_{\mathbb{R}^+}^{}\varphi_{xt}\varphi_x{\rm d}x+\frac{\mu}{2a}\int_{\mathbb{R}^+}^{}\bar\rho\zeta_x^2{\rm d}x\right\}+\frac{\mu D}{a}\int_{\mathbb{R}^+}^{}\bar\rho\zeta_{xx}^2{\rm d}x	\notag\\
			&+\int_{\mathbb{R}^+}^{}p^{\prime}(\bar\rho)\varphi_{xx}^2{\rm d}x-2\mu\int_{\mathbb{R}^+}^{}\bar\rho\zeta_x\varphi_{xx}{\rm d}x+\frac{b\mu}{a}\int_{\mathbb{R}^+}^{}\bar\rho\zeta_x^2{\rm d}x\notag\\
			\lesssim&(\varepsilon+\delta)(\|\varphi_{xx}\|^2+\|\zeta_{xx}\|^2)+\|\varphi_{xt}\|^2+{\rm e}^{-ct}\|\varphi_{xx}\|^2+(\varepsilon+\delta)(1+t)^{-1}\left(\|\varphi_x\|^2+\|\zeta\|_1^2\right)\notag\\
			&+\delta (1+t)^{-\frac{5}{2}}.
		\end{align}

Next, calculating $\int_{\mathbb{R}^+}\left(\partial_x\eqref{29}_1\times\varphi_{xt}\right){\rm d}x+\int_{\mathbb{R}}\left(\partial_x\eqref{29}_2\times\left(\frac{\mu}{a}\bar\rho\zeta_{xt}\right)\right){\rm d}x$ and using \eqref{13.19} show
		\begin{align}\label{13.55}
			&\;\;\;\;\frac{{\rm d}}{{\rm d}t}\left\{\frac{1}{2}\int_{\mathbb{R}^+}\varphi_{xt}^2{\rm d}x	+\frac{\mu D}{2a}\int_{\mathbb{R}^+}^{}\bar\rho\zeta_{xx}^2{\rm d}x\right\}+\alpha\int_{\mathbb{R}^+}\varphi_{xt}^2{\rm d}x+\frac{\mu}{a}\int_{\mathbb{R}^+}^{}\bar\rho\zeta_{xt}^2{\rm d}x\notag\\
			&\;\;\;\;+\frac{{\rm d}}{{\rm d}t}\left\{\frac{1}{2}\int_{\mathbb{R}^+}p^{\prime}(\bar\rho)\varphi_{xx}^2{\rm d}x-\mu\int_{\mathbb{R}^+}\bar\rho\zeta_x \varphi_{xx}{\rm d}x+\frac{\mu b}{2a}\int_{\mathbb{R}^+}^{}\bar\rho\zeta_x^2{\rm d}x\right\}\notag\\
			&=\frac{1}{2}\int_{\mathbb{R}^+}p^{\prime\prime}(\bar\rho)\bar\rho_t\varphi_{xx}^2{\rm d}x-\int_{\mathbb{R}^+}(p^{\prime\prime}(\bar\rho)\bar\rho_x\varphi_x)_x\varphi_{xt}{\rm d}x-\mu\int_{\mathbb{R}^+}\bar\rho_t\zeta_x \varphi_{xx}{\rm d}x+\frac{\mu b}{2a}\int_{\mathbb{R}^+}^{}\bar\rho_t\zeta_x^2{\rm d}x	\notag\\
			&\;\;\;\;+\frac{\mu D}{2a}\int_{\mathbb{R}^+}^{}\bar\rho_t\zeta_{xx}^2{\rm d}x
			-\frac{\mu D}{a}\int_{\mathbb{R}^+}^{}\bar\rho_x\zeta_{xx}\zeta_{xt} {\rm d}x-\frac{\mu}{b}\int_{\mathbb{R}^+}[\bar\rho_x(a\varphi_x-b\zeta)]_x\varphi_{xt}{\rm d}x\notag\\
			&\;\;\;\;	-\mu\int_{\mathbb{R}^+}(\varphi_x\zeta_x)_x\varphi_{xt}{\rm d}x-\mu\int_{\mathbb{R}^+}(\varphi_x\hat\phi_x)_x\varphi_{xt}{\rm d}x-\mu\int_{\mathbb{R}^+}(\hat\rho\zeta_x)_x\varphi_{xt}{\rm d}x-\mu\int_{\mathbb{R}^+}(\bar\rho\hat\phi_x)_x\varphi_{xt}{\rm d}x\notag\\
			&\;\;\;\;-\mu\int_{\mathbb{R}^+}(\hat\rho\bar\phi_x)_x\varphi_{xt}{\rm d}x-\mu\int_{\mathbb{R}^+}(\hat\rho\hat\phi_x)_x\varphi_{xt}{\rm d}x+\frac{\mu}{a}\int_{\mathbb{R}^+}^{}\bar\rho\zeta_{xt}g_x {\rm d}x-\int_{\mathbb{R}^+}f_{xx}\varphi_{xt}{\rm d}x-\int_{\mathbb{R}^+}h_{xx}\varphi_{xt}{\rm d}x\notag\\
			&\;\;\;\;+\int_{\mathbb{R}^+}p^{\prime\prime}(\bar\rho)\bar\rho_x\varphi_{xx}\varphi_{xt}{\rm d}x-\mu\int_{\mathbb{R}^+}\bar\rho_{xx}\zeta \varphi_{xt}{\rm d}x-\mu\int_{\mathbb{R}^+}\bar\rho_x\zeta_x\varphi_{xt}{\rm d}x=:\sum_{i=15}^{33}J_i,
		\end{align}
		where we used 
		\begin{align}
		\int_{\mathbb{R}^+}(\bar\rho\zeta)_{xx}\varphi_{xt}{\rm d}x=&-\int_{\mathbb{R}^+}(\bar\rho\zeta)_x\varphi_{xxt}{\rm d}x=-\int_{\mathbb{R}^+}(\bar\rho_x\zeta+\bar\rho\zeta_x)\varphi_{xxt}{\rm d}x\notag\\ 
		=&-\frac{{\rm d}}{{\rm d}t}\int_{\mathbb{R}^+}\bar\rho\zeta_x\varphi_{xx}{\rm d}x+\int_{\mathbb{R}^+}(\bar\rho_t\zeta_x+\bar\rho\zeta_{xt})\varphi_{xx}{\rm d}x+\int_{\mathbb{R}^+}(\bar\rho_{xx}\zeta+\bar\rho_{x}\zeta_x)\varphi_{xt}{\rm d}x.\notag	
		\end{align}
		We now estimate the terms on the right side of \eqref{13.55} one by one. Using Lemmas \ref{l3.1}--\ref{l3.2}, \eqref{13.15}, \eqref{13.18} and Young inequality yield
		\begin{align}\label{13.56}
			\sum_{i=15}^{21}J_i&\lesssim
			\delta(1+t)^{-1}\left(\|\varphi_{xx}\|^2+\|\zeta_x\|^2+\|\zeta_{xx}\|^2\right)+\delta(1+t)^{-2}\left(\|\varphi_{x}\|^2+\|\zeta\|^2\right)\notag\\
			&\;\;\;\;+\delta\left(\|\varphi_{xt}\|^2+\|\zeta_{xt}\|^2\right),
		\end{align}
		\begin{align}\label{13.57}
			J_{22}
			&\lesssim\|\zeta_x\|_{L^{\infty}}\|\varphi_{xx}\|\|\varphi_{xt}\|+\|\varphi_x\|_{L^{\infty}}\|\zeta_{xx}\|\|\varphi_{xt}\|\notag\\
			&\lesssim \varepsilon\|\varphi_{xt}\|^2+\varepsilon(1+t)^{-1}(\|\varphi_{xx}\|^2+\|\zeta_{xx}\|^2),
		\end{align}
	and
		\begin{align}\label{13.58}
			\sum_{i=23}^{27}J_i\lesssim\delta{\rm e}^{-ct}.
		\end{align}
		For $J_{28}$--$J_{33}$, we recall the definition of $g$ and $f$ in \eqref{22}, and apply Lemmas \ref{l3.1}--\ref{l3.2} and \eqref{13.15} to get
		\begin{align}\label{13.59}
			J_{28}
			&=\frac{\mu}{a}\frac{{\rm d}}{{\rm d}t}\int_{\mathbb{R}^+}\bar\rho\zeta_x g_x{\rm d}x-\frac{\mu}{a}\int_{\mathbb{R}^+}\bar\rho_t\zeta_x g_x{\rm d}x-\frac{\mu}{a}\int_{\mathbb{R}^+}\bar\rho\zeta_x g_{xt}{\rm d}x\notag\\
			&\leq \frac{\mu}{a}\frac{{\rm d}}{{\rm d}t}\int_{\mathbb{R}^+}\bar\rho\zeta_x g_x{\rm d}x+C(1+t)^{-1}\|\zeta_x\|^2+C(1+t)\int_{\mathbb{R}^+}(\bar\rho^2g_{xt}^2+\bar\rho_t^2g_x^2){\rm d}x\notag\\
			&\leq \frac{\mu}{a}\frac{{\rm d}}{{\rm d}t}\int_{\mathbb{R}^+}\bar\rho\zeta_x g_x{\rm d}x+C(1+t)^{-1}\|\zeta_x\|^2+C\delta(1+t)^{-\frac{7}{2}},
		\end{align}
		\begin{align}\label{13.60}
			J_{29}+J_{30}
			\leq& -\frac{1}{2}\frac{{\rm d}}{{\rm d}t}\int_{\mathbb{R}^+}[p^{\prime}(\varphi_x+\bar\rho+\hat\rho)-p^{\prime}(\bar\rho)]\varphi_{xx}^2{\rm d}x+C\varepsilon\|\varphi_{xt}\|^2+C(\varepsilon+\delta)(1+t)^{-1}\|\varphi_{xx}\|^2\notag\\
			&+C(\varepsilon+\delta)(1+t)^{-2}\left(\|\varphi_{x}\|^2+\|\varphi_t\|^2\right)+C{\rm e}^{-ct}(\|\varphi_{xx}\|_1^2+\|\varphi_{xt}\|^2)+C\delta(1+t)^{-\frac{7}{2}},
		\end{align}
		and
		\begin{align}\label{13.61}
			J_{31}+J_{32}+J_{33}&
			\lesssim\delta\|\varphi_{xt}\|^2+\delta(1+t)^{-1}\left(\|\varphi_{xx}\|^2+\|\zeta_{x}\|^2\right)+\delta(1+t)^{-2}\|\zeta\|^2.
		\end{align}
		Putting \eqref{13.56}--\eqref{13.61} into \eqref{13.55}, and using the smallness of $\varepsilon+\delta$, we can deduce
		\begin{align}\label{13.62}
			&\frac{{\rm d}}{{\rm d}t}\left\{\frac{1}{2}\int_{\mathbb{R}^+}\varphi_{xt}^2{\rm d}x	+\frac{\mu D}{2a}\int_{\mathbb{R}^+}^{}\bar\rho\zeta_{xx}^2{\rm d}x\right\}+\frac{\alpha}{2}\int_{\mathbb{R}^+}\varphi_{xt}^2{\rm d}x+\frac{\mu}{2a}\int_{\mathbb{R}^+}^{}\bar\rho\zeta_{xt}^2{\rm d}x\notag\\
			&+\frac{{\rm d}}{{\rm d}t}\left\{\frac{1}{2}\int_{\mathbb{R}^+}p^{\prime}(\bar\rho)\varphi_{xx}^2{\rm d}x-\mu\int_{\mathbb{R}^+}\bar\rho\zeta_x \varphi_{xx}{\rm d}x+\frac{\mu b}{2a}\int_{\mathbb{R}^+}^{}\bar\rho\zeta_x^2{\rm d}x\right\}\notag\\
			\leq&
			-\frac{1}{2}\frac{{\rm d}}{{\rm d}t}\int_{\mathbb{R}^+}[p^{\prime}(\varphi_x+\bar\rho+\hat\rho)-p^{\prime}(\bar\rho)]\varphi_{xx}^2{\rm d}x+\frac{\mu}{a}\frac{{\rm d}}{{\rm d}t}\int_{\mathbb{R}^+}\bar\rho\zeta_x g_x{\rm d}x\notag\\
			&+ C(\varepsilon+\delta)(1+t)^{-1}\left(\|\varphi_{xx}\|^2+\|\zeta_x\|_1^2\right)+C(\varepsilon+\delta)(1+t)^{-2}\left(\|\varphi_{x}\|^2+\|\varphi_t\|^2+\|\zeta\|^2\right)\notag\\ 
			&+C{\rm e}^{-ct}(\|\varphi_{xx}\|_1^2+\|\varphi_{xt}\|^2)+C\delta(1+t)^{-\frac{7}{2}}.
		\end{align}	
By calculating $(1 + t)^k \cdot (\lambda \cdot\eqref{13.54} + \eqref{13.62}) ~(0 < \lambda\ll 1)$ with $k = 0, 1$, respectively, then integrating it over $(0, t)$ and using Lemma \ref{l3.3}, we have	
		\begin{align}\label{13.63}
			&\;\;\;\;(1+t)\|[\varphi_x,\varphi_{xt},\varphi_{xx},\zeta_x,\zeta_{xx}](t)\|^2+\int_{0}^{t}(1+\tau)\|[\varphi_{x\tau},\zeta_{x\tau},\varphi_{xx},\zeta_x,\zeta_{xx}](\tau)\|^2{\rm d}\tau\notag\\
			&\leq C(\|\varphi_0\|_3^2+\|\psi_0\|_2^2+\|\zeta_0\|_3^2+\delta)+C\int_{0}^{t}{\rm e}^{-c\tau}\left(\|\varphi(\tau)\|_3^2+\|\varphi_{\tau}(\tau)\|_1^2\right){\rm d}\tau.
		\end{align}
	By calculating $(1 + t)^2 \cdot\eqref{13.62}$, then integrating it over $(0, t)$ and using Lemma \ref{l3.3} and \eqref{13.63} give
	\begin{align}\label{13.64}
			&\;\;\;\;(1+t)^2\|[\varphi_{xt},\varphi_{xx},\zeta_x,\zeta_{xx}](t)\|^2+\int_{0}^{t}(1+\tau)^2\|[\varphi_{x\tau},\zeta_{x\tau}](\tau)\|^2{\rm d}\tau\notag\\
			&\leq C(\|\varphi_0\|_3^2+\|\psi_0\|_2^2+\|\zeta_0\|_3^2+\delta)+C\int_{0}^{t}{\rm e}^{-c\tau}\left(\|\varphi(\tau)\|_3^2+\|\varphi_{\tau}(\tau)\|_1^2\right){\rm d}\tau,
		\end{align}
		which, together with \eqref{13.63} lead to \eqref{13.46}.	
	\end{proof}

	\begin{lemma}\label{l3.5}
		Under the hypothesis of Proposition \ref{p3.1}, it holds that
		\begin{align}\label{13.65}
			&\;\;\;\;(1+t)^3\|[\varphi_{tt},\varphi_{xt},\zeta_{t},\zeta_{xt}](t)\|^2+(1+t)^2\|\varphi_t(t)\|^2\notag\\
			&\;\;\;\;+\int_{0}^{t}[(1+\tau)^3\|[\varphi_{\tau\tau},\zeta_{\tau\tau}](\tau)\|^2+(1+\tau)^2\|[\varphi_{x\tau},\zeta_{x\tau},\zeta_{\tau}](\tau)\|^2]{\rm d}\tau\notag\\
			&\leq C(\|\varphi_0\|_3^2+\|\psi_0\|_2^2+\|\zeta_0\|_4^2+\delta)+C\int_{0}^{t}{\rm e}^{-c\tau}\left(\|\varphi(\tau)\|_3^2+\|\varphi_{\tau}(\tau)\|_1^2\right){\rm d}\tau,
		\end{align}
		and
		\begin{align}\label{13.66}
			&\;\;\;\;(1+t)^4\|[\varphi_{xtt},\varphi_{xxt},\zeta_{xt},\zeta_{xxt}](t)\|^2+(1+t)^3\|\varphi_{xt}(t)\|^2\notag\\
			&\;\;\;\;\int_{0}^{t}[(1+\tau)^4\|[\varphi_{x\tau\tau},\zeta_{x\tau\tau}](\tau)\|^2+(1+\tau)^3\|[\varphi_{xx\tau},\zeta_{x\tau},\zeta_{xx\tau}](\tau)\|^2]{\rm d}\tau\notag\\
			&\leq C(\|\varphi_0\|_3^2+\|\psi_0\|_2^2+\|\zeta_0\|_4^2+\delta)+C\int_{0}^{t}{\rm e}^{-c\tau}\left(\|\varphi(\tau)\|_3^2+\|\varphi_{\tau}(\tau)\|_1^2\right){\rm d}\tau.
		\end{align}
	\end{lemma}
\begin{proof}
Similar to the previous calculation, firstly, it follows from
$\int_{\mathbb{R}^+}(\partial_t\eqref{29}_1\times \varphi_{t}){\rm d}x+\int_{\mathbb{R}^+}(\partial_t\eqref{29}_2\times\frac{\mu}{a}\bar\rho\zeta_{t}){\rm d}x$	that 
\begin{align}\label{33}
			&\frac{{\rm d}}{{\rm d}t}\left\{\frac{\alpha}{2}\int_{\mathbb{R}^+}\varphi_t^2{\rm d}x+\int_{\mathbb{R}^+}^{}\varphi_{tt}\varphi_t{\rm d}x+\frac{\mu}{2a}\int_{\mathbb{R}^+}^{}\bar\rho\zeta_t^2{\rm d}x\right\}+\frac{\mu D}{a}\int_{\mathbb{R}^+}^{}\bar\rho\zeta_{xt}^2{\rm d}x	\notag\\
			&+\int_{\mathbb{R}^+}^{}p^{\prime}(\bar\rho)\varphi_{xt}^2{\rm d}x-2\mu\int_{\mathbb{R}^+}^{}\bar\rho\zeta_t\varphi_{xt}{\rm d}x+\frac{b\mu}{a}\int_{\mathbb{R}^+}^{}\bar\rho\zeta_t^2{\rm d}x\notag\\
			=&\int_{\mathbb{R}^+}\varphi_{tt}^2{\rm d}x-\mu\int_{\mathbb{R}^+}^{}\partial_t(\varphi_{x}\zeta_x)\varphi_{t} {\rm d}x-\frac{\mu}{b}\int_{\mathbb{R}^+}^{}\partial_t[\bar\rho_x(a\varphi_{x}-b\zeta)]\varphi_{t}{\rm d}x-\mu\int_{\mathbb{R}^+}^{}\partial_t(\varphi_{x}\hat\phi_x)\varphi_{t}{\rm d}x\notag\\
			&-\mu\int_{\mathbb{R}^+}^{}\partial_t(\hat\rho\zeta_x)\varphi_{t}{\rm d}x	-\mu\int_{\mathbb{R}^+}^{}\partial_t(\bar\rho\hat\phi_x)\varphi_{t}{\rm d}x-\mu\int_{\mathbb{R}^+}^{}\partial_t(\hat\rho\bar\phi_x)\varphi_{t}{\rm d}x	-\mu\int_{\mathbb{R}^+}^{}\partial_t(\hat\rho\hat\phi_x)\varphi_{t}{\rm d}x\notag\\
			&
			+\int_{\mathbb{R}^+}^{}f_{t}\varphi_{xt}{\rm d}x-\int_{\mathbb{R}^+}^{}p^{\prime}(\bar\rho)_t\varphi_{x}\varphi_{xt}{\rm d}x
			-\int_{\mathbb{R}^+}^{}h_{xt}\varphi_{t}{\rm d}x
			+\mu\int_{\mathbb{R}^+}^{}\bar\rho_t\zeta\varphi_{xt}{\rm d}x\notag\\ 
			&+\frac{\mu}{2a}\int_{\mathbb{R}^+}^{}\bar\rho_t\zeta_t^2{\rm d}x-\frac{\mu D}{a}\int_{\mathbb{R}^+}^{}\bar\rho_x\zeta_{xt}\zeta_{t} {\rm d}x+\frac{\mu}{a}\int_{\mathbb{R}^+}^{}\bar\rho\zeta_{t}g_t {\rm d}x\notag\\
			\lesssim& \|\varphi_{tt}\|_1^2 +\|\varphi_{xxt}\|^2+(\varepsilon+\delta)(\|\varphi_{xt}\|^2+\|\zeta_{t}\|_1^2)+(1+t)^{-1}\left(\|\varphi_{xx}\|^2+\|\varphi_t\|^2\right)\notag\\ 
			&+(1+t)^{-2}\left(\|\varphi_x\|^2+\|\zeta\|^2\right)+\delta (1+t)^{-\frac{7}{2}}. 
		\end{align}
Then, 
by $[\int_{\mathbb{R}^+}(\partial_t\eqref{29}_1\times \varphi_{tt}){\rm d}x+\int_{\mathbb{R}^+}(\partial_t\eqref{29}_2\times\frac{\mu}{a}\bar\rho\zeta_{tt}){\rm d}x$, we can also obtain that
\begin{align}\label{421k1}
			&\frac{{\rm d}}{{\rm d}t}\left\{\frac{1}{2}\int_{\mathbb{R}^+}\varphi_{tt}^2{\rm d}x	+\frac{\mu D}{2a}\int_{\mathbb{R}^+}^{}\bar\rho\zeta_{xt}^2{\rm d}x\right\}+\alpha\int_{\mathbb{R}^+}\varphi_{tt}^2{\rm d}x+\frac{\mu}{a}\int_{\mathbb{R}^+}^{}\bar\rho\zeta_{tt}^2{\rm d}x\notag\\
			&+\frac{{\rm d}}{{\rm d}t}\left\{\frac{1}{2}\int_{\mathbb{R}^+}p^{\prime}(\bar\rho)\varphi_{xt}^2{\rm d}x-\mu\int_{\mathbb{R}^+}\bar\rho\zeta_t \varphi_{xt}{\rm d}x+\frac{\mu b}{2a}\int_{\mathbb{R}^+}^{}\bar\rho\zeta_t^2{\rm d}x\right\}\notag\\
			=&\frac{1}{2}\int_{\mathbb{R}^+}p^{\prime\prime}(\bar\rho)\bar\rho_t\varphi_{xt}^2{\rm d}x-\mu\int_{\mathbb{R}^+}\bar\rho_t\zeta_t \varphi_{xt}{\rm d}x+\frac{\mu b}{2a}\int_{\mathbb{R}^+}^{}\bar\rho_t\zeta_t^2{\rm d}x	
			+\frac{\mu D}{2a}\int_{\mathbb{R}^+}^{}\bar\rho_t\zeta_{xt}^2{\rm d}x
			-\frac{\mu D}{a}\int_{\mathbb{R}^+}^{}\bar\rho_x\zeta_{xt}\zeta_{tt} {\rm d}x\notag\\
			&-\frac{\mu}{b}\int_{\mathbb{R}^+}[\bar\rho_x(a\varphi_x-b\zeta)]_t\varphi_{tt}{\rm d}x-\mu\int_{\mathbb{R}^+}(\varphi_x\zeta_x)_t\varphi_{tt}{\rm d}x-\mu\int_{\mathbb{R}^+}(\varphi_x\hat\phi_x)_t\varphi_{tt}{\rm d}x-\mu\int_{\mathbb{R}^+}(\hat\rho\zeta_x)_t\varphi_{tt}{\rm d}x\notag\\
			&-\mu\int_{\mathbb{R}^+}(\bar\rho\hat\phi_x)_t\varphi_{tt}{\rm d}x-\mu\int_{\mathbb{R}^+}(\hat\rho\bar\phi_x)_t\varphi_{tt}{\rm d}x-\mu\int_{\mathbb{R}^+}(\hat\rho\hat\phi_x)_t\varphi_{tt}{\rm d}x+\frac{\mu}{a}\int_{\mathbb{R}^+}^{}\bar\rho\zeta_{tt}g_t {\rm d}x-\int_{\mathbb{R}^+}h_{xt}\varphi_{tt}{\rm d}x\notag\\
			&-\int_{\mathbb{R}^+}f_{xt}\varphi_{tt}{\rm d}x-\int_{\mathbb{R}^+}(p^{\prime}(\bar\rho)_t\varphi_{x})_x\varphi_{tt}{\rm d}x+\mu\int_{\mathbb{R}^+}\bar\rho_{t}\zeta \varphi_{xtt}{\rm d}x\notag\\
			\leq&\frac{\alpha}{4}\|\varphi_{tt}\|^2+C(\varepsilon+\delta)\|\zeta_{tt}\|^2+C(1+t)\|\varphi_{xtt}\|^2+\frac{\mu}{a}\frac{{\rm d}}{{\rm d}t}\int_{\mathbb{R}^+}\bar\rho\zeta_t g_t{\rm d}x+C(1+t)^{-1}\left(\|\varphi_{xt}\|^2+\|\zeta_t\|_1^2\right)\notag\\ 
			&+C(1+t)^{-2}\left(\|\varphi_{xx}\|^2+\|\varphi_t\|^2+\|\zeta_x\|^2\right)+C(1+t)^{-3}\left(\|\varphi_x\|^2+\|\zeta\|^2\right)+\delta (1+t)^{-\frac{9}{2}}.
		\end{align}
Combining the above two formulas and using Lemmas \ref{l3.3}--\ref{l3.4}, we can get
\begin{align}\label{59}
			&\;\;\;\;(1+t)^2\|\varphi_t(t)\|^2+(1+t)^3\|[\varphi_{tt},\varphi_{xt},\zeta_{t},\zeta_{xt}](t)\|^2\notag\\
			&\;\;\;\;+\int_{0}^{t}[(1+\tau)^2\|[\varphi_{x\tau},\zeta_{x\tau},\zeta_{\tau}](\tau)\|^2+(1+\tau)^3\|[\varphi_{\tau\tau},\zeta_{\tau\tau}](\tau)\|^2]{\rm d}\tau\notag\\
			&\leq C(\|\varphi_0\|_3^2+\|\psi_0\|_2^2+\|\zeta_0\|_3^2+\delta)+C\int_{0}^{t}[(1+\tau)^2\|\varphi_{xx\tau}\|+(1+\tau)^4\|\varphi_{x\tau\tau}\|^2]{\rm d}\tau\notag\\ 
			&\;\;\;\;+C\int_{0}^{t}{\rm e}^{-c\tau}\left(\|\varphi(\tau)\|_3^2+\|\varphi_{\tau}(\tau)\|_1^2\right){\rm d}\tau.
		\end{align}
Then by calculating $\int_{\mathbb{R}^+}(\partial_{xt}\eqref{29}_1\times \varphi_{xt}){\rm d}x+\int_{\mathbb{R}^+}(\partial_{xt}\eqref{29}_2\times\frac{\mu}{a}\bar\rho\zeta_{xt}){\rm d}x$ and $\int_{\mathbb{R}^+}(\partial_{xt}\eqref{29}_1\times \varphi_{xtt}){\rm d}x+\int_{\mathbb{R}^+}(\partial_{xt}\eqref{29}_2\times\frac{\mu}{a}\bar\rho\zeta_{xtt}){\rm d}x$, respectively, we can prove
\begin{align}\label{13.70}
			&\frac{{\rm d}}{{\rm d}t}\left\{\frac{\alpha}{2}\int_{\mathbb{R}^+}\varphi_{xt}^2{\rm d}x+\int_{\mathbb{R}^+}^{}\varphi_{xtt}\varphi_{xt}{\rm d}x+\frac{\mu}{2a}\int_{\mathbb{R}^+}^{}\bar\rho\zeta_{xt}^2{\rm d}x\right\}+\frac{\mu D}{a}\int_{\mathbb{R}^+}^{}\bar\rho\zeta_{xxt}^2{\rm d}x	\notag\\
			&+\int_{\mathbb{R}^+}^{}p^{\prime}(\bar\rho)\varphi_{xxt}^2{\rm d}x-2\mu\int_{\mathbb{R}^+}^{}\bar\rho\zeta_{xt}\varphi_{xxt}{\rm d}x+\frac{b\mu}{a}\int_{\mathbb{R}^+}^{}\bar\rho\zeta_{xt}^2{\rm d}x\notag\\
			\lesssim& \|\varphi_{xtt}\|^2 +(\varepsilon+\delta)(\|\varphi_{xxt}\|^2+\|\zeta_{xt}\|_1^2)+(\varepsilon+\delta)(1+t)^{-1}(\|\varphi_{xt}\|^2+\|\zeta_{t}\|_1^2)+(1+t)^{-2}\left(\|\varphi_{xx}\|^2+\|\varphi_t\|^2\right)\notag\\ 
			&+(1+t)^{-3}\left(\|\varphi_x\|^2+\|\zeta\|^2\right)+\delta (1+t)^{-\frac{9}{2}}, 
		\end{align}
\begin{align}\label{13.71}
			&\frac{{\rm d}}{{\rm d}t}\left\{\frac{1}{2}\int_{\mathbb{R}^+}\varphi_{xtt}^2{\rm d}x	+\frac{\mu D}{2a}\int_{\mathbb{R}^+}^{}\bar\rho\zeta_{xxt}^2{\rm d}x\right\}+\alpha\int_{\mathbb{R}^+}\varphi_{xtt}^2{\rm d}x+\frac{\mu}{a}\int_{\mathbb{R}^+}^{}\bar\rho\zeta_{xtt}^2{\rm d}x\notag\\
			&+\frac{{\rm d}}{{\rm d}t}\left\{\frac{1}{2}\int_{\mathbb{R}^+}p^{\prime}(\bar\rho)\varphi_{xxt}^2{\rm d}x-\mu\int_{\mathbb{R}^+}\bar\rho\zeta_{xt} \varphi_{xxt}{\rm d}x+\frac{\mu b}{2a}\int_{\mathbb{R}^+}^{}\bar\rho\zeta_{xt}^2{\rm d}x\right\}\notag\\
			\leq&\frac{\alpha}{4}\|\varphi_{xtt}\|^2+(\varepsilon+\delta)\|\zeta_{xtt}\|^2+\frac{\mu}{a}\frac{{\rm d}}{{\rm d}t}\int_{\mathbb{R}^+}\bar\rho\zeta_{xt} g_{xt}{\rm d}x+(\varepsilon+\delta)(1+t)^{-1}(\|\varphi_{xxt}\|^2+\|\zeta_{xt}\|_1^2)\notag\\ 
			&+(\varepsilon+\delta)(1+t)^{-2}\left(\|\varphi_{xt}\|^2+\|\zeta_t\|_1^2\right)+(1+t)^{-3}\left(\|\varphi_{xx}\|^2+\|\varphi_t\|^2+\|\zeta_x\|^2\right)\notag\\ 
			&+(1+t)^{-4}\left(\|\varphi_x\|^2+\|\zeta\|^2\right)+\delta (1+t)^{-\frac{11}{2}}.
		\end{align}
Combining \eqref{59}--\eqref{13.71} and using Lemmas \ref{l3.3}--\ref{l3.4} give \eqref{13.65}--\eqref{13.66}. The proof of this lemma is completed.
	\end{proof}

\begin{lemma}\label{lem53}
		Under the hypothesis of Proposition \ref{p3.1}, it holds that
		\begin{align}\label{13.72}
			&(1+t)^2\|[\varphi_{xxx},\zeta_{xxx}(t)]\|^2\notag\\ 
			\leq& C(\|\varphi_0\|_3^2+\|\psi_0\|_2^2+\|\zeta_0\|_4^2+\delta)+C\int_{0}^{t}{\rm e}^{-c\tau}\left(\|\varphi(\tau)\|_3^2+\|\varphi_{\tau}(\tau)\|_1^2\right){\rm d}\tau.
		\end{align}
	\end{lemma}
\begin{proof}
We rewrite \eqref{18} as
\begin{align}\notag
	&\displaystyle\left(\frac{(-\varphi_{t}+\bar m+\hat m)^2}{\varphi_x+\bar\rho+\hat\rho}\right)_{x}+\left[p(\varphi_x+\bar\rho+\hat\rho)-p(\bar\rho)\right]_{x}\notag\\
			&=\varphi_{tt}+\mu \varphi_x(\zeta_x+\bar\phi_x+\hat\phi_x)+\mu\bar\rho(\zeta_x+\hat\phi_x)+\mu\hat\rho(\zeta_x+\bar\phi_x+\hat\phi_x)+\alpha \varphi_{t}-\bar m_t,\notag
	\end{align}	
	\begin{align}
	\zeta_t=D\zeta_{xx}+a\varphi_x-b\zeta-\bar\phi_t+D\bar\phi_{xx}+D\hat\phi_{xx}.\notag
	\end{align}	
By direct calculation, we have
\begin{align}\label{13.73}
(1+t)^2\|\varphi_{xxx}\|^{2}\lesssim&(1+t)^2(\|\varphi_{xtt}\|^{2}+\|\varphi_{xt}\|_1^{2}+\|\zeta_{xx}\|^2+(1+t)^{-2}\|\varphi_{x}\|^2+(1+t)^{-1}\|\varphi_{xx}\|^2+\delta(1+t)^{-\frac{5}{2}})\notag\\
\lesssim& \|\varphi_0\|_3^2+\|\psi_0\|_2^2+\|\zeta_0\|_4^2+\delta+\int_{0}^{t}{\rm e}^{-c\tau}\left(\|\varphi(\tau)\|_3^2+\|\varphi_{\tau}(\tau)\|_1^2\right){\rm d}\tau,
\end{align}
\begin{align}\label{13.74}
(1+t)^2\|\zeta_{xxx}\|^{2}\lesssim&(1+t)^2(\|\varphi_{xt}\|^{2}+\|\varphi_{xx}\|^{2}+\|\zeta_{x}\|^2+\delta(1+t)^{-3})\notag\\
\lesssim& \|\varphi_0\|_3^2+\|\psi_0\|_2^2+\|\zeta_0\|_4^2+\delta+\int_{0}^{t}{\rm e}^{-c\tau}\left(\|\varphi(\tau)\|_3^2+\|\varphi_{\tau}(\tau)\|_1^2\right){\rm d}\tau.
\end{align}
Thus, the proof of this lemma is completed.
\end{proof}
	
{\bf Proof of Proposition \ref{p3.1}:} From \eqref{43}, \eqref{13.46}, \eqref{13.65}--\eqref{13.66} and \eqref{13.72}, we can deduce that
\begin{align}
			\|(\varphi,\zeta)(t)\|_3^2+\|\varphi_t(t)\|_2^2\leq C(\|\varphi_0\|_3^2+\|\psi_0\|_2^2+\|\zeta_0\|_4^2+\delta)+C\int_{0}^{t}{\rm e}^{-c\tau}\left(\|\varphi(\tau)\|_2^2+\|\varphi_{\tau}(\tau)\|_1^2\right){\rm d}\tau.\notag
		\end{align}
Then according to Gronwall's inequality, we can get
\begin{align}\label{13.75}
			\sup_{0\leq t\leq T}\|(\varphi,\zeta)(t)\|_3^2+\|\varphi_t(t)\|_2^2\leq C(\|\varphi_0\|_3^2+\|\psi_0\|_2^2+\|\zeta_0\|_4^2+\delta).
		\end{align}
		Then combining \eqref{43}, \eqref{13.46}, \eqref{13.65}--\eqref{13.66}, \eqref{13.72} and \eqref{13.75}, we obtain \eqref{13.16}--\eqref{13.17}.
This completes the proof of Proposition \ref{p3.1}.

	\section{Proof of Theorems \ref{thm2}--\ref{thm3}}\label{s4}
	In this section, we will prove the Theorems \ref{thm2}--\ref{thm3}. For the sake of brevity, we only prove Theorem \ref{thm2}. The proof of Theorem \ref{thm3} can be obtained similarly.

Firstly, we study some fundamental properties of the nonlinear diffusion wave $\bar{\rho}(x,t)$. 
From \cite{Duyn-Peletier1977}, one can confirm that \eqref{1.2} has a unique self-similar solution in the following form:
	\[
	\bar\rho(x,t)=\bar\rho\left(\frac{x}{\sqrt{1+t}}\right):=\bar\rho(\xi),\;\;\;\xi\in \mathbb{R}^+,\;\;\;\mbox{and}\;\;\;\bar\rho(+\infty)=\rho_{+}.
	\]
	Similar to \cite{Hsiao-Liu1992,Nishihara-Yang1999}, we can prove that the self-similar solution satisfies
	\[
	\sum_{k=1}^{5}\left|\frac{d^k}{d\xi^k}\bar\rho(\xi)\right|+|\bar\rho(\xi)-\rho_+|+|\bar\rho-\bar\rho_{0}(0)|\leq C|\rho_+-\bar\rho_{0}(0)|\mbox{exp}\left(-c\alpha\xi^2\right).
	\]
	With these analyses, we get the following lemma.
	\begin{lemma}\label{l4.1}
		For any $1\leq p\leq +\infty$ and $ j+k\geq 1$,
		the self-similar solution $\bar\rho(x,t)$ of \eqref{1.2} satisfies
		\begin{equation}\label{14.1}
			\left\{
			\begin{array}{lr}
				\min\{\rho_0(0),\rho_+\}\leq \bar\rho(x,t)\leq \max\{\rho_0(0),\rho_+\},\\[2mm]
				\|\partial_t^j\partial_x^k\bar\rho(t)\|_{L^p}\leq C|\rho_+-\rho_0(0)|(1+t)^{-\frac{k}{2}-j+\frac{1}{2p}}.
			\end{array}
			\right.
		\end{equation}
	\end{lemma}
	\begin{remark}
		From  \eqref{1.2a}, one has that $ j+k\geq 1$, 
		\begin{equation}\label{14.2}
			\left\{
			\begin{array}{lr}
				\|\partial_t^j\partial_x^k\bar m(t)\|_{L^p}\leq C|\rho_+-\rho_0(0)|(1+t)^{-\frac{1}{2}-\frac{k}{2}-j+\frac{1}{2p}},\\[2mm]
				\|\partial_t^j\partial_x^k\bar\phi(t)\|_{L^p}\leq C|\rho_+-\rho_0(0)|(1+t)^{-\frac{k}{2}-j+\frac{1}{2p}}.
			\end{array}
			\right.
		\end{equation}
	\end{remark}
	
	From the definitions of $(\hat\rho,\hat m,\hat\phi)(x,t)$ in \eqref{3.8}--\eqref{3.8a}, we can immediately obtain the following lemma.
	\begin{lemma}\label{l4.2}
		For $1\leq p\leq +\infty$, and $k,j \in \mathbb{N}$, then we have
		\begin{equation}\label{14.3}
			\left\{
			\begin{array}{lr}
				\|\partial_x^k\partial_t^j(\hat\rho, \hat m_x, \hat\phi_x)(t)\|_{L^p}\lesssim \varepsilon_{0}^{k+1-\frac{1}{p}}{\rm e}^{-ct},\\[2mm]
				\|\partial_t^j(\hat m,\hat\phi)(t)\|_{L^{\infty}}\lesssim{\rm e}^{-ct}.
			\end{array}
			\right.
		\end{equation}
	\end{lemma}
	With these preparations, similar to the Proposition \ref{p3.1}, we can prove the following proposition:
	\begin{proposition}\label{p3.2}
		Under the condition of Theorem \ref{thm2},  and further assume that $\delta:=\left|\rho_{0}(0)-\rho_+\right|+\varepsilon_0^{\frac{1}{2}}$ and $\|\varphi_{0}\|_{3}+\|\psi_{0}\|_{2}+\|\zeta_{0}\|_{4}$ are sufficiency small. 
		Then for any given $T>0$, there are positive constants $\varepsilon$ and $C$ such that if the solution $(\varphi, \zeta)(x, t)$ to \eqref{20}--\eqref{22} on $0\leq t\leq T$ satisfies
		\begin{align}\label{13.15}
		&\sup_{0\leq t\leq T}\left\{\sum_{k=0}^{2}(1+t)^{k}\|\partial_x^k\varphi(t)\|^2+\sum_{k=0}^{2}(1+t)^{k+2}\|\partial_x^k\varphi_t(t)\|^2+\sum_{k=0}^{1}(1+t)^{k+1}\|\partial_x^k[\zeta,\zeta_x](t)\|^2\notag\right.\\ 
		&\ \ \ \ \ \ \ \ \ \ \ \ \ \ \ \ \ \ \ \ \ \ \ \ \ \ \ \ \left.+\sum_{k=0}^{1}(1+t)^{k+3}\|\partial_x^k[\zeta_t,\zeta_{xt}](t)\|^2+(1+t)^{2}\|(\partial_x^3\varphi,\partial_x^3\zeta)(t)\|^2\right\}\leq \varepsilon^2,\notag
		\end{align}
		then we have
		\begin{align}\notag
			&\;\;\;\;\sum_{k=0}^{2}(1+t)^{k}\|\partial_x^k\varphi(t)\|^2+\sum_{k=0}^{1}(1+t)^{k+1}\|\partial_x^k[\zeta,\zeta_x](t)\|^2+\sum_{k=0}^{2}\int_{0}^{t}(1+\tau)^k\|\partial_x^k[\varphi_x,\zeta,\zeta_x](\tau)\|^2{\rm d}\tau\notag\\
			&\leq C(\|\varphi_0\|_3^2+\|\psi_0\|_2^2+\|\zeta_0\|_4^2+\delta),\notag
		\end{align}
		and
		\begin{align}
			&\;\;\;\;\sum_{k=0}^{2}(1+t)^{k+2}\|\partial_x^k\varphi_t(t)\|^2+\sum_{k=0}^{1}(1+t)^{k+3}\|\partial_x^k[\zeta_t,\zeta_{xt}](t)\|^2+(1+t)^{2}\|(\partial_x^3\varphi,\partial_x^3\zeta)(t)\|^2\notag\\
			&\;\;\;\;+\sum_{k=0}^{2}\int_{0}^{t}(1+\tau)^{k+1}\|\partial_x^k\varphi_{\tau}(\tau)\|^2{\rm d}\tau+\sum_{k=0}^{1}\int_{0}^{t}(1+\tau)^{k+2}\|\partial_x^k[\zeta_{\tau},\zeta_{x\tau}](\tau)\|^2{\rm d}\tau\notag\\
			&\leq C(\|\varphi_0\|_3^2+\|\psi_0\|_2^2+\|\zeta_0\|_4^2+\delta).\notag
		\end{align}
	\end{proposition}
	With this proposition, we can immediately get Theorem \ref{thm2}.

	\vspace{6mm}

	 {\bf Conflict of Interest:} The authors declared that they have no conflict of interest.

\vspace{3mm}
	{\bf Acknowledgments} The research was supported by the National Natural Science Foundation of China \#12401283 and the Natural Science Foundation of Hubei Province \#2024AFB208.
	
	\bigbreak
	
	{\small
		\bibliographystyle{plain}
		
	}
\end{document}